\documentclass[12pt]{article}

\usepackage{amssymb}
\usepackage{latexsym, amsmath, amscd,amsthm}

\usepackage{amsmath}
\usepackage{amssymb}
\usepackage{multirow}
\usepackage{amsthm}

\newtheorem{theorem}{Theorem}[section]
\newtheorem{corollary}[theorem]{Corollary}
\newtheorem{proposition}[theorem]{Proposition}
\newtheorem{lemma}[theorem]{Lemma}

\theoremstyle{definition}
\newtheorem{definition}[theorem]{Definition}
\newtheorem{example}[theorem]{Example}

\newcommand{\ZZ}{{\mathbb Z}}
\newcommand{\ff}{{\mathbb F}}

\def\Res{\mathop{\operatorfont Res}\nolimits}

\newcommand{\Cc}{{\mathcal C}}
\newcommand{\Cd}{{\mathcal D}}
\def\Res{\mathop{\operatorfont res}\nolimits}
\def\Pic{\mathop{\operatorfont Pic}\nolimits}

\begin{document}

\title{Coset bounds for algebraic geometric codes}

\author{Iwan M. Duursma\footnotemark[1]\  
~and~ Seungkook Park\footnotemark[2]} 

\date{October 8, 2008}

\renewcommand{\thefootnote}{\fnsymbol{footnote}}
\footnotetext[1]{Department of Mathematics, University of Illinois
at Urbana-Champaign (duursma@math.uiuc.edu)}
\footnotetext[2]{Department of Mathematical Sciences, University of
Cincinnati (seung-kook.park@uc.edu)}
\renewcommand{\thefootnote}{\arabic{footnote}}

\maketitle

\begin{abstract}
For a given curve $X$ and divisor class $C$, we give lower bounds on the degree of a divisor
$A$ such that $A$ and $A-C$ belong to specified semigroups of divisors. For suitable choices
of the semigroups we obtain (1) lower bounds for the size of a party $A$ that can recover
the secret in an algebraic geometric linear secret sharing scheme with adversary threshold $C$,
and (2) lower bounds for the support $A$ of a codeword in a geometric Goppa code with 
designed minimum support $C$. Our bounds include and improve
both the order bound and the floor bound. The bounds are illustrated for two-point codes
on general Hermitian and Suzuki curves. 
\end{abstract}

\section*{Introduction} \label{S:Introduction}

Two recent results motivated this paper. The first is the complete description of the minimum distance  
of Hermitian two-point codes by Homma and Kim \cite{HomKim06}. The second is the introduction of algebraic 
geometric linear secret sharing schemes by Chen and Cramer \cite{CheCra06}. \\
For algebraic geometric codes,  the actual value of the minimum distance is not a priori known and needs to be determined 
or estimated from the data used in the construction. 
The best known lower bounds for the minimum distance of an algebraic geometric code are the order bound and the floor bound. 
Beelen \cite{Bee07FF}, and independently the second author \cite{SKP07}, have shown that the order bound agrees, for Hermitian two-point codes, with
the actual minimum distances found by Homma and Kim. In this paper we improve both the order bound and the floor bound. 
We illustrate our results and the obtained improvements for two-point codes from the Suzuki curves. \\ 
An important application of secret sharing schemes is secure multi-party computation,
which requires linear secret sharing schemes with a multiplicative property \cite{CraDamMau00}, \cite{CraetSix05}.
Chen and Cramer proposed to use one-point algebraic geometric codes for secret sharing and they have shown that the obtained algebraic geometric 
linear secret sharing schemes can be used for efficient secure computation over small fields \cite{CheCra06}. Parties can reconstruct 
a secret uniquely from their shares only if the total number of shares exceeds the adversary threshold of the secret sharing scheme. 
The algebraic geometric construction of a linear secret sharing scheme guarantees a lower bound for the adversary threshold.
The precise value of the threshold is in general not known. We show that the adversary threshold corresponds to the minimum
distance between cosets of a code. Our results give improved lower bounds for distances between cosets of an
algebraic geometric code, and therefore improved lower bounds for adversary thresholds of algebraic geometric linear secret sharing schemes. \\

As our main results, we formulate an \emph{ABZ bound for codes} and an \emph{ABZ bound for cosets}.
The bounds improve and generalize the floor bound and the order bound, respectively. 
For each of the bounds, we illustrate the improvements with examples from the Suzuki curves. The 
bounds can be used as tools for constructing improved codes as well as improved secret sharing schemes.
Our \emph{Main theorem} is an even more general bound. Its main advantage is that it has a short proof 
and that all other bounds can be obtained as special cases.  \\
The floor bound is independent of the order bound. Algorithms are available for decoding up to half the order bound 
but not for decoding up to half the floor bound. Beelen \cite{Bee07FF} gives an example where the floor bound exceeds the
order bound. For our generalizations there is a strict hierarchy. The improved order bound, obtained with the
ABZ bound for cosets, is at least the ABZ bound for codes, which improves the floor bound. We show that decoding is 
possible up to half the bound in our main theorem, and therefore up to half of all our bounds. In particular, we obtain
for the first time an approach to decode up to half the floor bound. \\
 
In Section \ref{S:CosetsLinearCodes}, we describe the use of linear codes for secret sharing
and the relation between coset distances and adversary thresholds. Theorem \ref{T:cosetbound}
gives a general coset bound for linear codes. Appendix \ref{S:CosetDecoding} gives a coset
decoding procedure that decodes up to half the bound. Algebraic geometric codes are defined in
Section \ref{S:AGCodes}. Theorem \ref{T:ABZcodes} gives the ABZ 
bound for algebraic geometric codes with a first proof based on the AB bound for linear codes.
Section \ref{S:CosetsAGCodes} gives a geometric characterization of coset distances for
algebraic geometric codes. 
In Section \ref{S:SemigroupIdeals} we define, for a divisor $C$ and for a point $P$,
a semigroup ideal
\[
\Gamma_P(C) = \{ A : L(A) \neq L(A-P) \wedge L(A-C) \neq L(A-C-P) \}
\]
such that the minimal degree for a divisor $A$ in $\Gamma_P(C)$
is a lower bound for the coset distance of an algebraic geometric code. 
In Section \ref{S:MainTheorem}, the main theorem gives a lower bound 
for the degree of a divisor in the semigroup ideal (Theorem \ref{T:cbdiv}). 
In Section \ref{S:OrderFloor}, we formulate the ABZ bound for cosets (Theorem \ref{T:cbabz}) 
and we describe its relation to both the order bound (Theorem \ref{T:order}) and the 
floor bound (Theorem \ref{T:floor}). 
The successful application of our bounds depends on the possibility to analyse the
complement 
\[
\Delta_P(C)  = \{ A : L(A) \neq L(A-P) \wedge L(A-C) = L(A-C-P) \} 
\]
and to compare naturally defined subsets of $\Delta_P(C)$.
Section \ref{S:DeltaSets} gives important basic relations among delta sets.  
In Section \ref{S:Discrepancies}, we define a discrepancy, for given 
points $P$ and $Q$, as a divisor $A \in \Delta_P(Q) = \Delta_Q(P).$ 
Discrepancies are our main tool for analyzing and improving lower bounds
for coset distances in large families of codes.
In Section \ref{S:Hermitian}, we give two proofs, 
one due to \cite{Bee07FF}, \cite{SKP07}, and one new, 
for lower bounds for the minimum distance of Hermitian two-point codes.
In Section \ref{S:Suzuki}, we determine discrepancies for Suzuki curves,
and we give examples of the ABZ bound for codes, the ABZ bound for cosets,
and the main theorem, that improve previously known bounds. 

\section{Cosets of linear codes} \label{S:CosetsLinearCodes}

Let $\ff$ be a finite field. A {$\ff$-linear code} $\Cc$ of {length} $n$ is a linear subspace of $\ff^n$. 
The {Hamming distance} between two vectors $x, y \in \ff^n$ is 
$d(x,y) = | \{ i : x_i \neq y_i \} |.$
The {minimum distance} of a nontrivial linear code $\Cc$ is
\begin{align*}
d(\Cc) &= \min\, \{ d(x,y) : x,y \in \Cc, x \neq y \} \\
      &= \min \, \{ d(x,0) : x \in \Cc, x \neq 0 \}.
\end{align*}
If $d(\Cc) \geq 2t+1$ and if $y \in \ff^n$ is at distance at most $t$ from $\Cc$ then there exists a unique
word $c \in \Cc$ with $d(c,y) \leq t.$   \\

The Hamming distance between two nonempty subsets $X, Y\subset \ff^n$ is the minimum of
$\{ d(x,y) : x \in X, y \in Y \}$. For a proper subcode $\Cc' \subset \Cc$, the
minimum distance of the collection of cosets $\Cc/\Cc'$ is
\begin{align*}
d(\Cc/\Cc') &=  \min\, \{ d(x+\Cc',y+\Cc') : x,y \in \Cc, x-y \not \in \Cc' \} \\
        &=  \min\, \{ d(x,0) : x \in \Cc , x \not \in \Cc' \}.
\end{align*} 

\begin{lemma}
If $d(\Cc/\Cc') \geq 2t+1$ and if $y \in \ff^n$ is at distance at most $t$ from $\Cc$ then there exists a unique
coset $c + \Cc' \in \Cc / \Cc'$ with $d(c+\Cc',y+\Cc') \leq t.$ 
\end{lemma}

The {dual code} $\Cd$ of $\Cc$ is the maximal subspace of $\ff^n$ that 
is orthogonal to $\Cc$ with respect to the standard inner product. To the
extension of codes $\Cc/\Cc'$ corresponds an extension of dual
codes $\Cd'/\Cd$ with distance parameter $d(\Cd'/\Cd).$ For two vectors
$x, y \in \ff^n$, let $x \ast y \in \ff^n$ denote the Hadamard or
coordinate-wise product of the two vectors. 

\begin{theorem}(Shift bound or Coset bound) \label{T:cosetbound}
Let $\Cc/\Cc_1$ be an extension of $\ff$-linear codes
with corresponding extension of dual codes $\Cd_1/\Cd$ such that 
$\dim \Cc/\Cc_1 = \dim \Cd_1/\Cd = 1$.
If there exist vectors $a_1, \ldots, a_w$ and $b_1, \ldots, b_w$ such that 
\[
\begin{cases} 
a_i \ast b_j \in \Cd    &\text{for $i+j \leq w$}, \\
a_i \ast b_j \in \Cd_1 \backslash \Cd &\text{for $i+j = w+1$}, 
\end{cases}
\]
then $d(\Cc/\Cc_1) \geq w.$ 
\end{theorem}

\begin{proof}
For all $c \in \Cc \backslash \Cc_1$ and $a \ast b \in \Cd_1 \backslash \Cd$,
$\sum_i a_i b_i c_i \neq 0$. To show the nonexistence of a vector $c \in \Cc \backslash \Cc_1$ 
with $d(c,0)<w$, it suffices to show, for any choice of $w-1$ coordinates, the existence of a vector 
$a \ast b \in \Cd_1 \backslash \Cd$ that is zero in those coordinates.
The conditions show that the vectors 
$a_1, \ldots, a_w$ are linearly independent, and there exists a nonzero linear combination 
$a$ of the vectors $a_1, \ldots, a_w$ vanishing at $w-1$ given coordinates. 
If $i$ is maximal such that $a_i$ has a nonzero coefficient in the linear combination $a$
then $a \ast b_{w+1-i} \in \Cd_1 \backslash \Cd$ is zero in the $w-1$ coordinates. 
\end{proof}

Let $y \in \ff^n$ be a word at distance at most $t$ from $\Cc$. For given vectors $a_1, \ldots, a_w$ 
and $b_1, \ldots, b_w$ such that $w > 2t$, the unique coset $c + \Cc_1 \in \Cc / \Cc_1$ with $d(c+\Cc_1,y+\Cc_1) \leq t$ 
can be computed efficiently with the coset decoding procedure in Appendix \ref{S:CosetDecoding}.
Theorem \ref{T:cosetbound} can be used to estimate the minimum distance $d(\Cc/\Cc')$ of an extension $\Cc/\Cc'$ 
with $\dim \Cc/\Cc' > 1,$ after dividing $\Cc/\Cc'$ into subextensions.
 
\begin{lemma} \label{T:itercosetbound}
Let $\Cc/\Cc'$ be an extension of $\ff$-linear codes of length $n$.
For $\Cc \supset \Cc'' \subset \Cc'$,
\[
d(\Cc/\Cc') = \min \{ d(\Cc/\Cc''), d(\Cc''/\Cc') \}. 
\]
\end{lemma}

We will now describe the use of code extensions for secret sharing. 
Our description focuses on the connection between secret sharing thresholds and coset distances
that will be established in Corollary \ref{C:ssbounds}. 
The main properties that we need are described in the following two lemmas. 

\begin{lemma} \label{L:dual}
Let $\{ 1, 2, \ldots, n \} = I \cup J$ be a partition of the coordinates. There exists
a word $r \in \Cc \backslash \Cc_1$ with support in $I$ if and only if there exists no
word $s \in \Cd_1 \backslash \Cd$ with support in $J$.
\end{lemma}

\begin{proof} 
Let $E_I$ (resp. $E_J$) be the subspace of $\ff^n$ of all vectors with support in $I$ 
(resp. $J$). The exact sequences 
\begin{align*}
0 \longrightarrow \Cc \cap E_I / \Cc_1 \cap E_I \longrightarrow &\Cc / \Cc_1 
  \longrightarrow \Cc + E_I / \Cc_1 + E_I \longrightarrow 0, \\
0 \longrightarrow \Cd_1 \cap E_J / \Cd \cap E_J \longrightarrow &\Cd_1 / \Cd 
  \longrightarrow \Cd_1 + E_J / \Cd + E_J \longrightarrow 0, 
\end{align*}
are in duality via $V \mapsto V^\ast = \text{\rm{Hom}}(V,\ff)$. And
\[
(\dim \Cc \cap E_I / \Cc_1 \cap E_I) + (\dim \Cd_1 \cap E_J / \Cd \cap E_J) = \dim \Cc / \Cc_1 = 1.
\]
\end{proof}

\begin{lemma} \label{L:DD1}
Let $y_1 \in \Cd_1 \backslash \Cd$. 
For a given vector $s \in \Cd_1 = \Cd \oplus \langle y_1 \rangle$, the projection of $s$ on
$\langle y_1 \rangle$ is uniquely determined by the subset of coordinates $\{ s_i : i \in A \}$ 
if and only if $\Cd_1 \backslash \Cd$ contains no word that is zero in the positions $A$.
\end{lemma}

\begin{proof}
The only if part is clear. For the if part we may assume with the previous lemma that there
exists $r \in \Cc \backslash \Cc_1$ with support in $A$. For any such $r$, and for 
$s = y + \lambda y_1$, $y \in \Cd$,
\[
r \cdot s = r \cdot (y + \lambda y_1) = \lambda (r \cdot y_1). 
\]
Since $(r \cdot y_1) \neq 0$, we obtain $\lambda = (r \cdot s) / (r \cdot y_1).$
\end{proof}

Let $y_1 \in \Cd_1 \backslash \Cd$. For a secret $\lambda \in \ff$, and for a random vector $y \in \Cd$, the vector
$s = y + \lambda y_1$ is called a vector of shares for $\lambda$. A subset
$A \subset \{1,2,\ldots,n\}$ is called qualified if the shares $\{ s_i : i \in A \}$ 
determine $\lambda$ uniquely. Whether $A$ is qualified depends on $\Cd_1 / \Cd$
but not on the vector of shares $s \in \Cd_1$. Let 
$\Gamma(\Cd_1 / \Cd)$ denote the collection of all subsets $A \subset \{1,2,\ldots,n\}$ 
that are qualified for $\Cd_1 / \Cd$ and let $\Delta(\Cd_1 / \Cd)$ denote the collection
of all subsets $A \subset \{1,2,\ldots,n\}$ that are not qualified for $\Cd_1 / \Cd$.
For the definition and main properties of a general linear secret sharing scheme 
we refer to \cite{CraetSix05}.

\begin{theorem} \label{T:ssbounds}
Let $\Cc/\Cc_1$ and $\Cd_1/\Cd$ be dual extensions of $\ff$-linear codes of length $n$.
Let $E_A$ be the subset of $\ff^n$ of all vectors with support in $A$.
\begin{align*}
\Gamma(\Cd_1 / \Cd) &= \{ A : \Cc \cap E_A \neq \Cc_1 \cap E_A \}, \\
\Delta(\Cd_1 / \Cd) &= \{ A : \Cc \cap E_A = \Cc_1 \cap E_A \}.
\end{align*}
Moreover, for ${\bar A} = \{1,2,\ldots,n\} \backslash A,$ 
\begin{align*}
\Gamma(\Cc / \Cc_1) &= \{ A : {\bar A} \in \Delta(\Cd_1 / \Cd) \}, \\
\Delta(\Cc / \Cc_1) &= \{ A : {\bar A} \in \Gamma(\Cd_1 / \Cd) \}.
\end{align*}
\end{theorem}

\begin{proof}
Lemma \ref{L:DD1} and Lemma \ref{L:dual}.
\end{proof}

\begin{corollary} \label{C:ssbounds}
The smallest qualified subset for $\Cd_1/\Cd$ is of size
\[
\min \{ |A| : A \in \Gamma(\Cd_1 / \Cd) \} = d(\Cc/\Cc_1).
\]
The largest unqualified subset for $\Cd_1/\Cd$ is of size
\[
\max \{ |A| : A \in \Delta(\Cd_1 / \Cd) \} = n-d(\Cd_1/\Cd),
\]
\end{corollary}

\section{Algebraic geometric codes} \label{S:AGCodes}

Let $X/\ff$ be an algebraic curve (absolutely irreducible, smooth,
projective) of genus $g$ over a finite field $\ff$. Let $\ff(X)$
be the function field of $X/\ff$ and let $\Omega(X)$ be the
module of rational differentials of $X/\ff$. 
Given a divisor $E$ on $X$ defined over $\ff$, let $L(E)$ denote the vector space over
$\ff$ of functions $f \in \ff(X) \backslash\{0\}$ with
$(f)+E\geq 0$ together with the zero function. Let $\Omega(E)$ denote the vector space over
$\ff$ of differentials $\omega \in \Omega(X) \backslash\{0\}$ with
$(\omega) \geq E$ together with the zero differential. Let $K$ represent
the canonical divisor class. \\

For $n$ distinct rational points $P_1, \ldots, P_n$ on $X$ and for disjoint divisors
$D=P_1+\cdots+P_n$ and $G$, the geometric Goppa codes $C_L(D,G)$ and $C_\Omega(D,G)$
are defined as the images of the maps
\begin{align*}
\alpha _L ~:~ &L(G)~~\longrightarrow~~\ff^{\,n}, ~~f \mapsto (\,f(P_1), \ldots, f(P_n) \,), \\
\alpha _\Omega ~:~ &\Omega(G-D)~~\longrightarrow~~\ff^{\,n}, 
~~\omega \mapsto (\,\Res_{P_1}(\omega), \ldots, \Res_{P_n}(\omega) \,).
\end{align*}
The maps establish isomorphisms $L(G)/L(G-D) \simeq C_L(D,G)$ and 
$\Omega(G-D)/\Omega(G) \simeq C_\Omega(D,G).$ With the Residue theorem, the images 
are orthogonal subspaces of $\ff^n$. With the Riemann-Roch theorem they are maximal
orthogonal subspaces. \\

There exists a nonzero word in $C_L(D,G)$ with support in $A$, for $0 \leq A \leq D$,
if and only if $L(G-D+A)/L(G-D) \neq 0.$ There exists a nonzero word in $C_\Omega(D,G)$ with support in $A$,
 for $0 \leq A \leq D$, if and only if $\Omega(G-A)/\Omega(G) \neq 0$ if and only if
$L(K-G+A) / L(K-G) \neq 0.$
  
\begin{proposition} \label{P:agd}
\begin{align*}
&d(C_L(D,G)) = \min \{ \deg A : 0 \leq A \leq D \;|\; L(A-C) \neq L(-C) \}, \quad \text{for $C = D-G.$} \\
&d(C_\Omega(D,G)) = \min \{ \deg A : 0 \leq A \leq D \;|\; L(A-C) \neq L(-C) \}, \quad \text{for $C = G-K.$} 
\end{align*}
\end{proposition}

\begin{theorem}(Goppa bound) \label{T:Goppa}
A nonzero word in $C_L(D,G)$ has weight $w \geq \deg\,(D-G)$.
A nonzero word in $C_\Omega(D,G)$ has weight $w \geq \deg\,(G-K)$.
\end{theorem}

The following bound improves on the Goppa bound in special cases (\cite{CarTor05}, \cite{MahMat06}, \cite{LunMcc06}).

\begin{theorem} \label{T:floor} (Floor bound)
Let $G = K+C = A+B+Z,$ for $Z \geq 0$ such that $L(A+Z) = L(A)$ and $L(B+Z) = L(B)$.
For $D$ with $D \cap Z = \emptyset$, a nonzero word in $C_\Omega(D,G)$ has weight 
at least $\deg\, C + \deg\, Z$.
\end{theorem}

Most algebraic bounds for the minimum distance of a linear code rely on one of 
two basic arguments. In the paper \cite{LinWil86} on cyclic codes they were named
the AB bound and the Shift bound. We obtain the following bound, which includes the
floor bound, using the AB bound argument in combination with the Goppa bound.

\begin{theorem} \label{T:ABZcodes} (ABZ bound for codes)
Let $G = K+C = A+B+Z,$ for $Z \geq 0$. For $D$ with $D \cap Z = \emptyset$, a 
nonzero word in $C_\Omega(D,G)$ has weight $w \geq l(A)-l(A-C)+l(B)-l(B-C).$ 
\end{theorem}

\begin{proof}
We may assume that $A$ and $B$ are disjoint from $D$.
Since $Z \geq 0$ and $D \cap Z = \emptyset$, the code $C_L(D,G)$ contains the code $C_L(D,A+B)$ as a subcode. 
Thus, for a word $c \in C_\Omega(D,G)$, and for words $a \in C_L(D,A)$ and $b \in C_L(D,B)$, if $c$ has support $D'$ then
$\sum_{P \in D'} a_P b_P c_P = 0.$ The last orthogonality holds for all $a \in C_L(D',A)$ and $b \in C_L(D',B)$,
so that $\dim C_L(D',A) + \dim C_L(D',B) \leq \deg D',$ and $\deg D' \geq l(A)-l(A-D') + l(B)-l(B-D').$
Together with $L(D'-C) \neq 0$, $\deg D' \geq l(A)-l(A-C) + l(B) - l(B-C).$
\end{proof}   

Example \ref{E:abz} gives a code for which the ABZ bound improves both the floor bound
and the order bound. It is easy to see, using the Riemann-Roch theorem, that the choice $Z=0$ returns the
Goppa bound. 
Improvements of the Goppa bound are obtained only if the divisors $A, B,$ and $Z,$ are carefully chosen. 
For the special case $L(A+Z)=L(A)$ and $L(B+Z)=L(B)$, we recover the floor
bound. In that case, for $K+C=A+B+Z$, 
\begin{align*}
  &l(A)-l(A-C)+l(B)-l(B-C) \\ 
=~ &l(A+Z)-l(K-B-Z)+l(B+Z)-l(K-A-Z) \\ 
=~ &\deg (A+Z) + \deg (B+Z) + 2 - 2g = \deg C + \deg Z.
\end{align*}

\section{Cosets of algebraic geometric codes} \label{S:CosetsAGCodes}

Let $\Cd = C_\Omega(D,G)$ and $\Cc = C_L(D,G)$ be dual algebraic geometric codes.
For a point $P$ disjoint from $D$, let
\begin{align*}
\Cd_1 / \Cd = C_\Omega(D,G-P) / C_\Omega(D,G), \\
\Cc / \Cc_1 = C_L(D,G) / C_L(D,G-P), 
\end{align*}
be dual extensions of codes. When $\dim \Cc / \Cc_1 = \dim \Cd_1 / \Cd = 1$,
the extensions can be used for secret sharing as described in Section \ref{S:CosetsLinearCodes}. 
Theorem \ref{T:ssbounds} describes the parties that can recover the secret for the
extension $\Cd_1 / \Cd$ as the subsets $0 \leq A \leq D$ that support a word in $\Cc / \Cc_1.$
The formulation in terms of divisors is given in Proposition \ref{P:O}, with a similar
result for the extension $\Cc / \Cc_1$ in Proposition \ref{P:L}. 
As additional motivation, we give a natural choice for the secret
for each of the extensions $\Cd_1 / \Cd$ and $\Cc / \Cc_1$, and we describe directly the 
qualified parties that can determine the secret, in Lemma \ref{L:O} and Lemma \ref{L:L},
respectively. The propositions can then also be obtained from the lemmas. \\

Let $P$ have multiplicity $e$ in $G$, and let $t$ be a fixed local parameter for $P$. 
For $\dim \Cd_1 / \Cd = 1$, there exists a natural isomorphism
$\Omega(G-D-P)/\Omega(G-D) \simeq \Cd_1 / \Cd \simeq \ff$ that maps $\omega \in \Omega(G-D-P)/\Omega(G-D)$ to
$\Res_P(t^{-e} \omega).$
For $\dim \Cc / \Cc_1 = 1,$ there exists a natural isomorphism $L(G)/L(G-P) \simeq \Cc / \Cc_1 \simeq \ff$,
that maps $f \in L(G)/L(G-P)$ to $(f t^e)(P).$ 

\begin{lemma} \label{L:O}
For $\omega \in \Omega(G-D-P)$, the residue $\Res_P(t^{-e}\omega)(P)$ is uniquely determined
by the values $\{ f(P) : P \in A \}$, for $0 \leq A \leq D,$ if and only if 
$\Omega(G-D+A-P) = \Omega(G-D+A)$. 
\end{lemma}

\begin{proposition} \label{P:O}
For the extension of codes $\Cd_1 / \Cd = C_\Omega(D,G-P) / C_\Omega(D,G)$, where $D=P_1+\cdots+P_n$ is a sum of 
$n$ distinct points, $G$ is a divisor disjoint from $D$, and $P$ is a point disjoint from $D$, 
\begin{align*}
\Gamma(\Cd_1 / \Cd) &= \{ 0 \leq A \leq D : \Cc \cap E_A \neq \Cc_1 \cap E_A \}, \\ 
                    &= \{ 0 \leq A \leq D :  L(G-D+A) \neq L(G-D+A-P) \}. \\
\Delta(\Cd_1 / \Cd) &= \{ 0 \leq A \leq D : \Cc \cap E_A = \Cc_1 \cap E_A \}, \\ 
                    &= \{ 0 \leq A \leq D :  L(G-D+A) = L(G-D+A-P) \}. 
\end{align*}
\end{proposition}
\begin{proof}
In each case, the two descriptions are clearly equivalent. The first description of $\Gamma(\Cd_1 / \Cd)$ uses Theorem \ref{T:ssbounds}. 
The second description uses Lemma \ref{L:O}.
\end{proof}

\begin{lemma} \label{L:L}
For $f \in L(G)$, the value $(t^e f)(P)$ is uniquely determined
by the values $\{ f(P) : P \in A \}$, for $0 \leq A \leq D,$ if and only if 
$L(G-A) = L(G-A-P)$.
\end{lemma}

\begin{proposition} \label{P:L}
For the extension of codes $\Cc / \Cc_1 = C_L(D,G) / C_L(D,G-P)$, where $D=P_1+\cdots+P_n$ is a sum of 
$n$ distinct points, $G$ is a divisor disjoint from $D$, and $P$ is a point disjoint from $D$, 
\begin{align*}
\Gamma(\Cc / \Cc_1) &= \{ 0 \leq A \leq D : \Cd_1 \cap E_A \neq \Cd \cap E_A \}, \\ 
                    &= \{ 0 \leq A \leq D :  \Omega(G-A-P) \neq \Omega(G-A) \}. \\
\Delta(\Cc / \Cc_1) &= \{ 0 \leq A \leq D : \Cd_1 \cap E_A = \Cd \cap E_A \}, \\ 
                    &= \{ 0 \leq A \leq D :  \Omega(G-A-P) = \Omega(G-A) \}. 
\end{align*}
\end{proposition}
\begin{proof}
As in Proposition \ref{P:O} but use Lemma \ref{L:L}.
\end{proof}

The propositions are related via the dualities $A \in \Gamma(\Cd_1 / \Cd)$ 
if and only if $D-A \in \Delta(\Cc / \Cc_1)$ and $A \in \Gamma(\Cc / \Cc_1)$ if and only if $D-A \in \Delta(\Cd_1 / \Cd)$
(as Theorem \ref{T:ssbounds}). The minimal degree of a divisor $A \in \Gamma(\Cd_1/\Cd)$
or $A \in \Gamma(\Cc/\Cc_1)$ is given by the coset distance $d(\Cc / \Cc_1)$ or $d(\Cd_1/\Cd)$,
respectively (as in Corollary \ref{C:ssbounds}).

\begin{proposition} \label{P:cb}
\begin{align*}
d(\Cc / \Cc_1) &= \min \{ \deg A :  0 \leq A \leq D \;|\; L(A-C) \neq L(A-C-P) \}, \quad \text{for $C = D-G,$} \\ 
d(\Cd_1 / \Cd) &= \min \{ \deg A :  0 \leq A \leq D \;|\; L(A-C) \neq L(A-C-P) \}, \quad \text{for $C = G-K-P.$} 
\end{align*}
\end{proposition}

For a given divisor $C$ and a point $P$, let 
\[
\Gamma_{P}(C) = \{ A : L(A) \neq L(A-P) \wedge L(A-C) \neq L(A-C-P) \},
\]
and let $\gamma_P(C)$ be the minimal degree for a divisor $A \in \Gamma_P(C).$ So that
$\gamma_P(C) \geq \max \{ 0, \deg C \}.$ 

\begin{theorem} \label{T:gammacoset}
For the extensions of codes $\Cd_1 / \Cd = C_\Omega(D,G-P) / C_\Omega(D,G)$ and 
$\Cc / \Cc_1 = C_L(D,G) / C_L(D,G-P)$,
\begin{align*}
A \in \Gamma(\Cd_1 / \Cd) &~\Rightarrow~ \deg A \,\geq\, \gamma_P(D-G) \,\geq\, n - \deg G. \\
A \in \Delta(\Cd_1 / \Cd) &~\Rightarrow~ \deg A \,\leq\, n - \gamma_P(G-K-P) \,\leq\, n - \deg G + 2g - 1. \\[1ex]
A \in \Gamma(\Cc / \Cc_1) &~\Rightarrow~ \deg A \,\geq\, \gamma_P(G-K-P) \,\geq\, \deg G - 2g +1. \\
A \in \Delta(\Cc / \Cc_1) &~\Rightarrow~ \deg A \,\leq\, n - \gamma_P(D-G) \,\leq\, \deg G. 
\end{align*} 
\end{theorem}

The lower bounds for $\deg A$ that are obtained with $\gamma_P(D-G)$ and $\gamma_P(G-K-P)$ use the assumption
$L(A) \neq L(A-P)$ instead of the stronger assumption $0 \leq A \leq D.$ Thus, when the bound for
$\deg A$ is not attained by divisors $A$ of the form $0 \leq A \leq D$, the bounds will not be optimal.
Essentially, we separate the problem of finding a small $A \in \Gamma(\Cd_1/\Cd)$ into two parts: 
a geometric part that considers all effective divisors $A$ not containing $P$, and an arithmetic part that verifies if $A$
can be represented by a divisor with $0 \leq A \leq D$. Only the first part is considered in this paper.
In other words, the bounds that we obtain apply to a different and more general problem, that of recovering 
local data at a point $P$ from given local data at a divisor $A$, for any divisor $A$ with no 
base point at $P$. We briefly outline this setting. 

\begin{definition} 
Let $X/\ff$ be a curve, and let $C$ be a divisor on $X$. For a given point $P$ on $X$ define
the collection $\Sigma_P(C) = \{ \pi_A : A \geq 0 \}$ of surjective maps
\[
\pi_A : \Omega(-C-P) \longrightarrow \Omega(-C-P)/\Omega(A-C-P), \quad \text{$A \geq 0.$}
\]
\end{definition}

The map $\pi_A$ assigns to a differential $\omega \in \Omega(-C-P)$ the local information
$\omega$ modulo $\Omega(A-C-P)$, in short the local information of $\omega$ at $A$.
Given that $\omega \in \Omega(-C-P)$, any sufficiently large amount
of local information determines $\omega$ uniquely. Indeed, for any divisor $A$ of 
sufficiently large degree, $\Omega(A-C-P) = 0$ and $\pi_A$ is a bijection. For a divisor $A$ 
with the weaker property $\Omega(A-C-P) = \Omega(A-C),$
the maps $\pi_A = \pi_{A+P}$ agree.  In that case, the local information of $\omega$ at $A$ 
determines uniquely the local information of $\omega$ at $A+P$. If $P$ occurs in the 
support of $A$ then this means that the local information can be determined with increased
precision. 
For secret sharing we assume that the secret corresponds to a fixed map $\pi_P$. 
Then the parties that do not know $\pi_P$ a priori are
those with $L(A) \neq L(A-P).$ Among those, the parties that can determine $\pi_P$ 
from $\pi_A$ are those that satisfy $\Omega(A-C-P) = \Omega(A-C)$, or, equivalently,
$L(A-C) \neq L(A-C-P).$ Together the conditions define the set $\Gamma_P(C).$
In this setting, the access structure $\Gamma_P(C)$ can be analysed without further
assumptions on the representation of the maps $\pi_A$. The image under $\pi_A$ of
a differential $\omega \in \Omega(-C-P)$ might 
be written out explicitly in terms of local parameters and residues, much like an
algebraic geometric code, or it might simply be given as a differential $\omega+\eta$
for $\eta \in \Omega(A-C-P).$ 

\section{Semigroup ideals} \label{S:SemigroupIdeals}

Let $X/\ff$ be a curve over a field $\ff$ and let $\Pic(X)$ be the group of divisor classes. 
Let $\Gamma = \{ A : L(A) \neq 0 \}$ be the semigroup of effective divisor classes.
For a given point $P \in X$, let $\Gamma_P = \{ A : L(A) \neq L(A-P) \}$ be the semigroup of effective divisor
classes with no base point at $P$. Call $A \in \Gamma_P$ a $P$-denominator for the divisor class $C \in \Pic(X)$ if 
$A-C \in \Gamma_P.$ So that $A-(A-C)$ expresses $C$ as the difference of two effective divisor 
classes without base point at $P$. The $P$-denominators for $C$ form the $\Gamma_P$-ideal
\[
\Gamma_P(C) = \{ A \in \Gamma_P : A-C \in \Gamma_P \}.
\]
The ideal structure of the semigroup $\Gamma_P(C)$ amounts to the property 
$A+E \in \Gamma_P(C)$ whenever $A \in \Gamma_P(C)$ and $E \in \Gamma_P.$
The $\Gamma_P$-ideal of $P$-numerators for $C$ is the ideal
\[
\Gamma_P(-C) = \{ A \in \Gamma_P : A+C \in \Gamma_P \}.
\]
Clearly, $A$ is a $P$-denominator for $C$ if and only if $A-C$ is a $P-$numerator for $C$, that is 
\[
A \in \Gamma_P(C) ~\Leftrightarrow~ A-C \in \Gamma_P(-C),
\]
The minimal degree $\gamma_P(C)$ of a $P$-denominator for $C$ is defined as
\[
\gamma_P(C) = \min \{ \deg A : A \in \Gamma_P(C) \}.
\]
The minimal degrees satisfy
\[ 
\gamma_P(C) - \gamma_P(-C) = \deg C.
\]
The denominator and numerator terminology is borrowed from the ideal interpretation of divisors.
Let $O$ be the ring of rational functions in $\ff(X)$ that are regular outside $P$.
For effective divisors $A$ and $B$ disjoint from $P$, the fractional $O$-ideal  
$\cup_{i \geq 0} L(i P-(B-A)) = J I^{-1}$ is the quotient of the integral $O$-ideals 
$J = \cup_{i \geq 0} L(i P - B)$ and $I = \cup_{i \geq 0} L(i P-A).$
To a denominator $A$ of smallest degree corresponds an ideal $I$ of smallest norm. \\

If either $C \in \Gamma_P$ or $-C \in \Gamma_P$ then the conditions $A \in \Gamma_P$
and $A-C \in \Gamma_P$ are dependent.

\begin{proposition} \label{P:gammap}
For a divisor $C$ on a curve $X$ of genus $g$,
$\gamma_P(C) \geq \max \{ 0, \deg C\}.$ 
Moreover, 
\begin{align*}
\gamma_P(C) = 0 ~\Leftrightarrow~ -C \in \Gamma_P ~&\Leftrightarrow~ \Gamma_P(C)  = \Gamma_P. \\
\gamma_P(C) = \deg C ~\Leftrightarrow~ C \in \Gamma_P ~&\Leftrightarrow~ \Gamma_P(-C)  = \Gamma_P.
\end{align*}
The inequality is strict if and only if $C, -C \not \in \Gamma_P$ only if $|\deg C| < 2g.$ 
\end{proposition}

For suitable choices of the divisor $C$, the parameter $\gamma_P(C)$ gives a lower bound for the coset 
distance of an algebraic geometric code (Proposition \ref{P:cb}) and therefore bounds for the access 
structure of an algebraic geometric linear secret sharing scheme (Theorem \ref{T:gammacoset}). 
Proposition \ref{P:gammap} shows that we can expect improvements over the trivial lower bound $\gamma_P(C) \geq \deg C$ 
that is used for Theorem \ref{T:gammacoset} only if $P$ is a base point for the divisor $C$. \\

Let $S$ be a finite set of rational points that includes $P$. For $\Gamma_S = \cap_{P \in S} \Gamma_P,$ let
$\Gamma_P(C;S) = \Gamma_P(C) \cap \Gamma_S = \{ A \in \Gamma_S : A-C \in \Gamma_P \}$, and let
$\gamma_P(C;S)$ be the minimal degree for a divisor $A \in \Gamma_P(C;S).$ 

\begin{lemma} \label{L:cosetgammas}
For a given set of rational points $S$ that includes $P$, 
and for extensions of algebraic geometric codes $C_\Omega(D,G-P) / C_\Omega(D,G)$ 
and $C_L(D,G) / C_L(D,G-P)$ defined with a divisor $D=P_1+\cdots+P_n$ disjoint from $S$, 
\begin{align*}
d(C_L(D,G) / C_L(D,G-P)) &\geq \gamma_P(C;S), \quad \text{for $C = D-G,$} \\ 
d(C_\Omega(D,G-P) / C_\Omega(D,G)) &\geq \gamma_P(C;S), \quad \text{for $C = G-K-P.$} 
\end{align*}
\end{lemma}
\begin{proof}
Proposition \ref{P:cb}.
\end{proof}

To obtain similar estimates for the minimum distance of an algebraic geometric code, we
use Proposition \ref{P:agd}. Define the $\Gamma_S$-ideals $\Gamma^\ast(C;S) \subseteq \Gamma(C;S)$, 
\begin{align*}
\Gamma^\ast(C;S) &= \{ A \in \Gamma_S : L(A-C) \neq L(-C) \}, \\
\Gamma(C;S) &= \{ A \in \Gamma_S : L(A-C) \neq 0 \}.
\end{align*}
Let $\gamma^\ast(C;S)$ (resp. $\gamma(C;S)$) denote the minimal degree for a divisor $A \in \Gamma^\ast(C;S)$ (resp. $A \in \Gamma(C,S)$).

\begin{lemma} \label{L:codegammas}
For a given set of rational points $S$, and for algebraic geometric codes $C_L(D,G)$ and $C_\Omega(D,G)$ defined
with a divisor $D=P_1+\cdots+P_n$ disjoint from $S$, 
\begin{align*}   
&d(C_L(D,G)) \geq \gamma^\ast(C;S) \geq \gamma(C;S), \quad \text{for $C=D-G$}, \\
&d(C_\Omega(D,G)) \geq \gamma^\ast(C;S) \geq \gamma(C;S), \quad \text{for $C=G-K$}. 
\end{align*}
For $L(-C) = 0$, $\gamma^\ast(C;S) = \gamma(C;S)$.
\end{lemma}
\begin{proof}
Proposition \ref{P:agd}.
\end{proof}

The condition $L(-C) = 0$ holds in all cases where the Goppa lower bound $d \geq \deg C$ (Theorem \ref{T:Goppa}) is positive.
We give lower bounds for $\gamma(C;S)$ using lower bounds for $\gamma_P(C;S)$. With a minor
modification, we obtain lower bounds for $\gamma^\ast(C;S).$ 

\begin{lemma} \label{L:gammas}
Let $S$ be a finite set of rational points. For a divisor $C$, and for a point $P \in S$,
\begin{align*}
&\Gamma(C;S) = \Gamma_P(C;S) \cup \Gamma(C+P;S). \\
&\Gamma^\ast(C;S) \subseteq \Gamma_P(C;S) \cup \Gamma^\ast(C+P;S). 
\end{align*}
Moreover, for $-C \in \Gamma_P$, 
\[
\Gamma^\ast(C;S) \subseteq \Gamma^\ast(C+P;S). 
\]
\end{lemma}
\begin{proof}
For the equality, $L(A-C) \neq 0$ if and only if $L(A-C) \neq L(A-C-P)$ or $L(A-C-P) \neq 0.$
For the inclusion, $L(A-C) \neq L(-C)$ only if $L(A-C) \neq L(A-C-P)$ or $L(A-C-P) \neq L(-C-P).$
Finally, for $A \in \Gamma^\ast(C;S)$ such that $-C \in \Gamma_P$, we have $\dim L(A-C) / L(-C-P) > 1$, and thus $L(A-C-P) \neq L(-C-P).$
So that $A \in \Gamma^\ast(C+P;S).$
\end{proof} 

\begin{proposition} \label{P:gammas}
\begin{align*}
\gamma(C;S) &\geq \min \{  \gamma_P(C;S), \gamma(C+P;S) \}. \\
\gamma^\ast(C;S) &\geq \min \{  \gamma_P(C;S), \gamma^\ast(C+P;S) \} \backslash \{ 0 \}. 
\end{align*}
\end{proposition}
\begin{proof}
In general $\gamma^\ast(C;S) > 0.$ And $\gamma_P(C;S) = 0$ only if $\gamma_P(C) = 0$ if and only if $-C \in \Gamma_P$,
in which case we can omit $\gamma_P(C;S)$ before taking the minimum.
\end{proof}

\section{Main theorem} \label{S:MainTheorem}

For a given curve $X / \ff$, let $C \in \Pic(X)$ be a divisor class and let $P$ be a point on $X$.
For the semigroup $\Gamma_P = \{ A : L(A) \neq L(A-P) \}$ and the
$\Gamma_P$-ideal 
\[
\Gamma_P(C) = \{ A \in \Gamma_P : A-C \in \Gamma_P \}, 
\]
define the complement 
\[
\Delta_P(C) = \{ A \in \Gamma_P : A-C \not \in \Gamma_P \}.
\]

\begin{lemma} \label{L:Dempty}
\[
\Delta_P(C) = \emptyset ~\Leftrightarrow~ \Gamma_P(C) = \Gamma_P ~\Leftrightarrow~ -C \in \Gamma_P. \\
\]
\end{lemma}

Let $X$ be of genus $g$ and let $K$ represent the canonical divisor class. 

\begin{lemma} \label{L:Ddegree}
In general,
\[
A \in \Delta_P(C) ~\Leftrightarrow~ K+C+P-A \in \Delta_P(C).
\]
For $A \in \Delta_P(C)$, 
\[
\min \{ 0, \deg C \} \;\leq\; \deg A \;\leq\; \max \{ 2g-1, \deg C + 2g-1 \}.
\]
\end{lemma}
\begin{proof}
This follows from the definition together with the Riemann-Roch theorem.
\end{proof}

The following is the analogue of Theorem \ref{T:cosetbound} in the language of divisors.

\begin{theorem} (Coset bound for divisors) \label{T:cbdiv}
Let $\{ A_1 \leq A_2 \leq \cdots \leq A_w \} \subset \Delta_P(C)$ be
a sequence of divisors with $A_{i+1} \geq A_i + P$, for $i=1,\ldots,w-1.$
Then $\deg A \geq w$, for every divisor $A \in \Gamma_P(C)$ with support disjoint
from $A_w-A_1,$ that is
\[
\gamma_P(C;A_w-A_1) \geq w.
\]
\end{theorem}
\begin{proof}
After replacing the sequence with an equivalent sequence if necessary, 
we may assume that $A_1, A_2, \ldots, A_w$ are disjoint from $A$. 
We obtain two sequences of subspaces.  
\begin{multline*}
L(A_w) \supsetneq L(A_w - P) \supseteq L(A_{w-1}) \supsetneq L(A_{w-1}-P) \supseteq \cdots \\ 
\cdots \supseteq L(A_2) \supsetneq L(A_2-P) \supseteq L(A_1) \supsetneq L(A_1-P).
\end{multline*}
\begin{multline*}
\Omega(A_w-C) \subsetneq \Omega(A_w-C-P) \subseteq \Omega(A_{w-1}-C) \subsetneq \Omega(A_{w-1}-C-P) \subseteq \cdots \\
\cdots \subset \Omega(A_2-C) \subsetneq \Omega(A_2-C-P) \subseteq \Omega(A_1-C) \subsetneq \Omega(A_1-C-P).
\end{multline*}
For $i=1,2,\ldots,w,$ choose
\[
f_i \in L(A_i) \backslash L(A_i-P) ~~\text{and}~~ \eta_i \in \Omega(A_i-C-P) \backslash \Omega(A_i-C).
\]
Let $A \in \Gamma_P$ be of degree $\deg A < w.$ Then there exists a linear combination
$f$ of $f_1, f_2, \ldots, f_w$ that vanishes on $A$. If $f_i$ is the leading function in the 
linear combination then $f \in L(A_i - A) \backslash L(A_i-A-P)$ and 
$f \eta_i \in \Omega(-C-P+A) \backslash \Omega(-C+A)$. Thus $A-C \not \in \Gamma_P$
and $A \not \in \Gamma_P(C).$
\end{proof} 

For a divisor $B$, let
\begin{align*}
\Delta_P(B,C) &= \{ B+iP : i \in \ZZ \} \cap \Delta_P(C), \\
              &= \{ B+iP \in \Gamma_P, B-C+iP \not \in \Gamma_P \}.
\end{align*}

\begin{lemma} \label{L:DminD}
To the set $\Delta_P(B,C)$ corresponds a dual set 
\begin{align*}
\Delta_P(B-C,-C) &= \{ B-C+iP : i \in \ZZ \} \cap \Delta_P(-C), \\
              &= \{ B-C+iP \in \Gamma_P, B+iP \not \in \Gamma_P \},
\end{align*}
such that $\# \Delta_P(B,C) - \# \Delta_P(B-C,-C) = \deg C.$
Furthermore,
\[
\# \Delta_P(B,C) = \begin{cases} \deg C, &\quad \text{if $C \in \Gamma_P.$} \\
                                0, &\quad \text{if $-C \in \Gamma_P.$}
\end{cases}
\]
In particular,
\[
\# \Delta_P(B,C) = \begin{cases} \deg C, &\quad \text{if $\deg C \geq 2g.$} \\
                                0, &\quad \text{if $\deg C \leq -2g.$}
\end{cases}
\]
\end{lemma}
\begin{proof}
For $i_0$ large enough, 
\begin{align*}
&\# \Delta_P(B,C) - \# \Delta_P(B-C,-C) \\
=~&\# \{ i \leq i_0 : B+iP \in \Gamma_P, B-C+iP \not \in \Gamma_P \} \\
 &\quad - \# \{ i \leq i_0 : B+iP \not \in \Gamma_P, B-C+iP \in \Gamma_P \} \\
=~&\sum_{i \leq i_0} (l(B+iP)-l(B+iP-P))-(l(B-C+iP)-l(B-C+iP-P)) \\
=~&\dim L(B+i_0 P) - \dim L(B-C+i_0 P) = \deg C.
\end{align*}
For the remainder use Lemma \ref{L:Dempty}.
\end{proof}

\begin{corollary} \label{C:cbfrac}
For any choice of divisor $B$, there is a pair of equivalent bounds
\begin{align*}
&\gamma_P(C) \geq \# \Delta_P(B,C).
&\gamma_P(-C) \geq \# \Delta_P(B-C,-C).
\end{align*}
\end{corollary}

\begin{proof}
For the first inequality, the elements $A_1, A_2, \ldots, A_w \in \Delta_P(B,C)$, ordered from lowest to highest degree,
meet the conditions of the theorem. Similar for the second inequality. Equivalence follows from 
$\gamma_P(C) - \gamma_P(-C) = \deg C$ and the previous lemma.
\end{proof}
 
\begin{lemma}
If $A \in \Gamma_P(E)$ and $E \in \Gamma_P(C)$ then $A \in \Gamma_P(C).$
For $E \in \Gamma_P(C)$, 
\[
\Delta_P(C) \subset \Delta_P(E).
\]
\end{lemma}
\begin{proof}
The first claim is immediate from the definitions, in particular $A-E \in \Gamma_P$
and $E-C \in \Gamma_P$ implies $A-C \in \Gamma_P.$ For $E \in \Gamma_P(C)$, the first
claim shows that $A \not \in \Gamma_P(E)$ whenever $A \not \in \Gamma_P(C).$
\end{proof}
 
\section{Order bound and floor bound} \label{S:OrderFloor}

We unify and improve two known lower bounds for the minimum distance of an algebraic geometric code. 
Let $S$ be a given set of rational points, and let $C_L(D,G)$ and $C_\Omega(D,G)$ be algebraic geometric codes
defined with a divisor $D=P_1+\cdots+P_n$ disjoint from $S$. With Lemma \ref{L:codegammas}, 
\begin{align*}   
&d(C_L(D,G)) \geq \gamma^\ast(C;S), \quad \text{for $C=D-G$}, \\
&d(C_\Omega(D,G)) \geq \gamma^\ast(C;S), \quad \text{for $C=G-K$}. 
\end{align*}
 
\begin{proposition} \label{P:order} 
For points $Q_0, \ldots, Q_{r-1} \in S$, define divisors 
$C_0 \leq C_1 \leq \cdots \leq C_r$ such that $C_0=C$ and $C_{i+1}=C_i+Q_i$, for $i=0,\ldots,r-2.$ Then
\[
\gamma^\ast(C;S) \geq\; \min \{ \gamma_{Q_0}(C_0;S), \gamma_{Q_1}(C_1;S), \ldots, \gamma_{Q_{r-1}}(C_{r-1};S), \gamma^\ast(C_r;S) \} \backslash \{0\}, \\
\]
In general, $\gamma^\ast(C_r;S) \geq \deg C + r.$
\end{proposition}
\begin{proof}
Proposition \ref{P:gammas} gives $\gamma^\ast(C_i;S) \geq \min \{ \gamma_{Q_i}(C_i;S), \gamma^\ast(C_{i+1};S) \} \backslash \{0\}.$ 
\end{proof}

We give a formulation of the order bound for an algebraic geometric code $C_\Omega(D,G).$ For the case that $G$ is supported
in two points, a similar result formulated in terms of near order functions can be found in \cite[Theorem 1]{Caretal07}. 

\begin{theorem} (Order bound \cite[Theorem 7]{Bee07FF})
Let ${\cal C}$ be an algebraic curve and $G$ a rational divisor. Let ${\cal P}$ be a set of rational points not occuring in the support of the 
divisor $G$. Then we have
\[
d ({C}_{\cal P}(G) ) \geq d_{\cal P}(G) \geq d(G).
\]
\end{theorem}

Using \cite[Remark 5, Definition 6]{Bee07FF}, we expand the theorem in the notation of the current paper. In comparison
with the original theorem, we have removed the condition that the divisors $B_0, \ldots, B_r$ are disjoint from $D$.

\begin{theorem} (Order bound \cite{Bee07FF}) \label{T:order}
Let $C_\Omega(D,G)$ be an algebraic geometric code, and let $G=K+C.$ For a sequence of points $Q_0, \ldots, Q_{r-1}$ disjoint from $D$, let
$C_0=C$ and $C_{i+1}=C_i+Q_i$, for $i=0,\ldots,r-2.$
\[
\Cc_0 = C_\Omega(D,K+C) \supseteq \Cc_1 = C_\Omega(D,K+C_1) \supseteq \cdots \supseteq \Cc_r = C_\Omega(D,K+C_r).
\]
If $\Cc_i \neq \Cc_{i+1}$ then a word in $\Cc_i \backslash \Cc_{i+1}$ has weight $w \geq \# \Delta_{Q_i}(0,C_i)$.
For $r$ large enough,
\[
d(C_\Omega(D,G)) \geq \min \{ \# \Delta_{Q_i}(0,C_i) : \Cc_i \neq \Cc_{i+1} \}. 
\]
Moreover, for a sequence of divisors $B_0, \ldots, B_{r-1}$,
\[
d(C_\Omega(D,G)) \geq \min \{ \# \Delta_{Q_i}(B_i,C_i) :  \Cc_i \neq \Cc_{i+1} \}.
\]
\end{theorem}
\begin{proof}
The order bound for the minimum distance combines Proposition \ref{P:order} with the
estimates $\gamma_{Q_i}(C_i;S) \geq \gamma_{Q_i}(C_i) \geq \Delta_{Q_i}(B_i,C_i)$ in Corollary \ref{C:cbfrac}.
\end{proof}  

We analyse the choice of the points $Q_0, Q_1, \ldots, Q_r$. In \cite{Bee07FF}, the choice of the points is
unrestricted, and an example is given where the optimal lower bound is obtained with a choice of $Q_i$ outside $G$.
On the other hand, Proposition \ref{P:gammap} shows that $\gamma_{Q_i}(C_i) \geq \deg C_i$. Thus, we may assume 
that the minimum $\min \{ \gamma_{Q_i}(C_i) \} \backslash \{0\}$ is taken over an interval
$i=0,1,\ldots,r$ such that, for all $i$ in the interval, either $\gamma_{Q_i}(C_i)=0$ or $\gamma_{Q_i}(C_i) > \deg C_i$.
With Proposition \ref{P:gammap} this implies that either $-C_i \in \Gamma_{Q_i}$ or $C_i \not \in \Gamma_{Q_i}.$
In both cases, we can conclude, for $C_i \neq 0$, that $C_i \not \in \Gamma_{Q_i}$, i.e. that $L(C_i) = L(C_i-Q_i).$
The same conclusion can be reached with Lemma \ref{L:DminD} if the argument is repeated for $\Delta_{Q_i}(B_i,C_i)$
instead of $\gamma_{Q_i}(C_i).$ The following stronger result holds.

\begin{proposition}
The maximum in the order bound is attained for a choice of points $Q_0, Q_1, \ldots, Q_{r-1}$ such that,
for $i=0,1,\ldots,r-1$, either $C_i=0$, or $Q_i, \ldots, Q_{r-1}$ are base points of the divisor $C_i$. 
In particular, if $C_i$ is a nonzero effective divisor, we may restrict the choice for $Q_i, \ldots, Q_{r-1}$
to points in the support of $C_i$. 
\end{proposition}
\begin{proof}
For $Q \in \{ Q_i, \ldots, Q_{r-1} \},$ let $j$ be minimal in $\{ i, \ldots, r-1 \}$ such that $Q_j=Q$. 
If $C_j \neq 0$, we may assume as explained above, that $C_j \not \in \Gamma_Q$. With $E = Q_i + \cdots + Q_{j-1} \in \Gamma_Q$
and $C_j=C_i+E$ it follows that $C_i \not \in \Gamma_Q$. If $C_j = 0$ then either $i=j$, in which case $C_i=0$, 
or $i<j$, in which case $\deg C_i < 0$ and $C_i \not \in \Gamma_Q.$
\end{proof}    

In \cite[Example 8]{Bee07FF}, the minimum distance lower bound for a code $C_\Omega(D,5P)$ on the Klein curve is improved with a choice 
$Q_0=P, Q_1=Q \neq P.$ For the example, $5P=K+2P-Q$ and $6P=K+Q+R$, so that $C_0=2P-Q$ and $C_1=Q+R.$ Indeed, with the proposition,
we can expect improvements only with $Q_1 = Q$ or with $Q_1 = R$. \\

To improve the order bound we apply the main theorem with a different format for the divisors
$A_1, \ldots, A_w.$ Let
\begin{align*}
\Delta_P(\leq B,C) &= \{ B+iP \in \Gamma_p : B-C-iP \not \in \Gamma_P \;\wedge\; i \leq 0\}, \\
\Delta_P(\geq B+P,C) &= \{ B+iP \in \Gamma_p : B-C-iP \not \in \Gamma_P \;\wedge\; i \geq 1 \},
\end{align*}
be a partition of the set $\Delta_P(B,C)$ into divisors of small and large degree.

\begin{lemma} \label{L:cnt}
\[
\# \Delta_P(\leq B,C) = \dim L(B) - \dim L(B-C) + \# \Delta_P(\leq B-C,-C).
\]
\end{lemma} 
\begin{proof}
Similar to the proof of Lemma \ref{L:DminD}, but use $i_0=0.$
\end{proof}

\begin{theorem} \label{T:cbabz} (ABZ bound for cosets)
Let $C$ be a divisor and let $P$ be a point. 
For $G=K+C=A+B+Z,$ $Z \geq 0$, 
\begin{align*}
\gamma_P(C;Z \cup P) &\geq \# \Delta_P(\leq A,C) + \# \Delta_P(\leq B,C). \\
\end{align*}
\end{theorem}
\begin{proof}
With Lemma \ref{L:Ddegree}, a divisor $A' \in \Delta_P(C)$ if and only if $K+C+P-A' \in \Delta_P(C)$.
And $A' \leq A$ if and only if $K+C+P-A' \geq K+C+P-A = B+P+Z.$ The
elements $A_1, A_2, \ldots, A_w \in \Delta_P(\leq B,C) \cup \Delta_P(\geq B+P+Z,C)$, ordered from lowest to highest degree,
meet the conditions of Theorem \ref{T:cbdiv}, with $w = \# \Delta_P(\leq A,C) + \# \Delta_P(\leq B,C).$
\end{proof}

The lower bound $\# \Delta_P(B,C)$ that is used for the order bound takes into account only the number of divisors in a delta set 
$\Delta_P(B,C)$. The improved bounds in Theorem \ref{T:cbabz} are possible by considering also the degree distribution 
of divisors in the delta set. For $Z=0$, the bounds in the theorem include those used in the order bound (Theorem \ref{T:order}). 
The floor bound (Theorem \ref{T:floor}) sometimes exceeds the order bound. The ABZ bound for
codes (Theorem \ref{T:ABZcodes}) gives an improvement and generalization of the floor bound. We show that the bounds in the theorem
not only include those obtained with the order bound but also those obtained with the ABZ bound for codes. 
In each case, the coset decoding procedure in the appendix decodes efficiently up to half the bound.

\begin{theorem} \label{T:Floor2} (ABZ bound for codes)
Let $G = K+C = A+B+Z,$ for $Z \geq 0$. For $D$ with $D \cap Z = \emptyset$, a nonzero word in 
$C_\Omega(D,G)$ has weight $w \geq l(A)-l(A-C)+l(B)-l(B-C).$ 
\end{theorem}
\begin{proof}
Let $P$ be a point on the curve not in the support of $D$, if necessary it can be chosen over an extension field.
We use Proposition \ref{P:order} with $S=Z \cup P$ and $Q_0=Q_1=\ldots=Q_{r-1}=P$. 
\[
\gamma^\ast(C;S) \geq \min \{ \gamma_{P}(C;S), \gamma_{P}(C+P;S), \ldots, \gamma_{P}(C+(r-1)P;S), \gamma^\ast(C+rP;S) \} \backslash \{0\}.
\]
Now use Theorem \ref{T:cbabz} with $K+C+iP = A+B+(Z+iP)$,
\[
\gamma_P(C+iP;S) \geq  \# \Delta_P(\leq A,C+iP) + \# \Delta_P(\leq B,C+iP). 
\]
With Lemma \ref{L:cnt},
\begin{align*}
\gamma_P(C+iP;S) &\geq  l(A)-l(A-C-iP)+l(B)-l(B-C-iP) \\
               &\geq l(A)-l(A-C)+l(B)-l(B-C).
\end{align*}
Hence, by taking $r$ large enough, $\gamma^\ast(C;S) \geq l(A)-l(A-C)+l(B)-l(B-C).$
\end{proof}

Neither the $ABZ$ bound for codes, nor the $ABZ$ bound for cosets gives an improvement in general. For $Z=0$, both
bounds return previously known bounds, namely the Goppa bound and the order bound, respectively. 
For carefully chosen nontrivial $Z$, there are possible improvements. If we apply Lemma \ref{L:cnt} with
both $A$ and $B$, 
\begin{align*}
\# \Delta_P(\leq A,C) = \dim L(A) - \dim L(A-C) + \# \Delta_P(\leq A-C,-C), \\
\# \Delta_P(\leq B,C) = \dim L(B) - \dim L(B-C) + \# \Delta_P(\leq B-C,-C),
\end{align*}
and add the two equations, then we see that the improvement of the ABZ coset bound
applied to $G=K+C=A+B+Z$ over the floor bound applied to $G=K+C=A+B+Z$ is given by the 
ABZ coset bound applied to the dual decomposition $G'=K-C=(A-C)+(B-C)+Z.$ For $Z=0$, we
recover that the improvement of the order bound applied to $G=K+C$ over the Goppa bound 
$\deg C$ is given by the order bound applied to $G'=K-C$ (Lemma \ref{L:DminD} and Corollary \ref{C:cbfrac}). \\
 
We consider the special case of the order bound with $B_0 = \cdots = B_{r-1} = 0$ and $Q_0 = \cdots = Q_{r-1} = P.$ 
For codes of the form $C_L(D,\rho P)^\perp = C_\Omega(D,\rho P)$ or of the form $C_L(D,K+P+\rho P)^\perp = C_\Omega(D,K+P+\rho P)$
the resulting bound can be formulated entirely in terms of the numerical semigroup $S$ of Weierstrass $P$-nongaps.
For the first code use $C = \rho P - K,$ and for the second $C = \rho P+P.$ For the delta sets we obtain 
\begin{align*}
p P \in \Delta_P(\rho P -K) &~\Leftrightarrow~ p P \in \Gamma_P \wedge K + p P - \rho P \not \in \Gamma_P. \\
                    &~\Leftrightarrow~ p \in S \wedge \rho - p + 1 \in S, \\
p P \in \Delta_P(\rho P ) &~\Leftrightarrow~ p P \in \Gamma_P \wedge p P - \rho P-P \not \in \Gamma_P. \\
                    &~\Leftrightarrow~ p \in S \wedge p-\rho-1 \not \in S.  
\end{align*} 
The first of the two bounds in the following theorem is the Feng-Rao bound \cite{FenRao93}, \cite{CamFarMun00}.
The second bound is different when the canonical divisor $K \not \sim (2g-2)P.$  

\begin{theorem}(Feng-Rao bound)
Let $S$ be the semigroup of Weierstrass $P$-nongaps. 
\[
d(C_L(D,\rho P)^\perp) \geq \min \{ \#A[\rho'] : \rho' > \rho \} \backslash \{ 0 \}, 
\]
where $A[\rho] = \{ p \in S | \rho - p \in S \}$.
\[
d(C_L(D,K + \rho P)^\perp) \geq \min \{ \#B[\rho'] : \rho' > \rho \} \backslash \{ 0 \},
\]
where $B[\rho] = \{ p \in S | p - \rho \not \in S \}$.
\end{theorem}
\begin{proof}
Apply Proposition \ref{P:order} with the given delta sets.
\end{proof}

\section{Delta sets} \label{S:DeltaSets}

For a divisor $C$ and a point $P$, we defined the $\Gamma_P$-ideal
$\Gamma_P(C) = \{ A \in \Gamma_P : A-C \in \Gamma_P \}.$
Theorem \ref{T:cbdiv} gives a lower bound for $\deg A$, for $A \in \Gamma_P(C)$, in terms of
the complement $\Delta_P(C) = \{ A \in \Gamma_P : A-C \not \in \Gamma_P \}$.
Theorem \ref{T:order} and Theorem \ref{T:cbabz} are formulated in terms of the subsets
$\Delta_P(B,C)$ and $\Delta_P(\leq B, C)$, respectively, for a suitable choice of divisor $B$. 
The computation of optimal lower bounds requires either a complete description of the delta set 
(for the main theorem) or at least a description from which the size of the sets $\Delta_P(B,C)$ or $\Delta_P(\leq B,C)$ can be
computed (for the other two theorems). 
We collect some straightforward relations that can be used to construct delta sets, to compare delta sets, or to compare
sizes of delta sets. Most relations come in pairs such that $A \in \Delta_P(C)$ (i.e., $A \in \Gamma_P, A-C \not \in \Gamma_P$) 
corresponds to $A-C \in \Delta_P(-C)$ (i.e., $A \not \in \Gamma_P, A-C \in \Gamma_P$). The proofs in this section are entirely 
straightforward, in most cases applying the definition of $\Delta_P(C)$ is enough, and no proofs are included.
In general, for $E \in \Gamma_P$, $\Delta_P(C) \subset \Delta_P(C+E).$ Lemma \ref{L:dset} gives a precise version and its dual.

\begin{lemma} \label{L:dset}
Let $C$ be a divisor and $P$ a point. For $E \in \Gamma_P$,
\begin{align*}
&A \in \Delta_P(C) \\
&\qquad ~\Leftrightarrow~ A \in \Delta_P(C+E) \;\wedge\; A+E \in \Delta_P(C+E). \\[1ex]
&A-C-E \in \Delta_P(-C-E) \\
&\qquad ~\Leftrightarrow~ A-C-E \in \Delta_P(-C) \;\wedge\; A-C \in \Delta_P(-C).
\end{align*}
\end{lemma}
For the four relations on the right we describe when the reverse implication fails. 
\begin{lemma} \label{L:diffA}
\begin{align*}
&A \in \Delta_P(C+E) \;\wedge\; A \not \in \Delta_P(C), \\
&\qquad ~\Leftrightarrow~ A\in \Delta_P(C+E) \;\wedge\; A+E \not \in \Delta_P(C+E), \\ 
&\qquad ~\Leftrightarrow~ A\in \Delta_P(C+E) \;\wedge\; A-C \in \Gamma_P, \\ 
&\qquad ~\Leftrightarrow~ A-C \in \Delta_P(E) \;\wedge\; A \in \Gamma_P. \\[1ex]
&A-C \in \Delta_P(-C) \;\wedge\; A-C-E \not \in \Delta_P(-C-E), \\
&\qquad ~\Leftrightarrow~ A-C \in \Delta_P(-C) \;\wedge\; A-C-E \not \in \Delta_P(-C), \\
&\qquad ~\Leftrightarrow~ A-C \in \Delta_P(-C) \;\wedge\; A-C-E \not \in \Gamma, \\ 
&\qquad ~\Leftrightarrow~ A-C \in \Delta_P(E) \;\wedge\; A \not \in \Gamma_P.
\end{align*}
\end{lemma}

The second group follows with a substitution $A \mapsto A-C-E, C \mapsto -C-E$.

\begin{lemma} \label{L:diffB}
\begin{align*}
&A-C-E \in \Delta_P(-C) \;\wedge\; A-C-E \not \in \Delta_P(-C-E), \\
&\qquad ~\Leftrightarrow~ A-C-E\in \Delta_P(-C) \;\wedge\; A-C \not \in \Delta_P(-C), \\ 
&\qquad ~\Leftrightarrow~ A-C-E\in \Delta_P(-C) \;\wedge\; A \in \Gamma_P, \\ 
&\qquad ~\Leftrightarrow~ A \in \Delta_P(E) \;\wedge\; A-C-E \in \Gamma_P. \\[1ex]
&A \in \Delta_P(C+E) \;\wedge\; A-E \not \in \Delta_P(C), \\
&\qquad ~\Leftrightarrow~ A \in \Delta_P(C+E) \;\wedge\; A-E \not \in \Delta_P(C+E), \\
&\qquad ~\Leftrightarrow~ A \in \Delta_P(C+E) \;\wedge\; A-E \not \in \Gamma, \\ 
&\qquad ~\Leftrightarrow~ A \in \Delta_P(E) \;\wedge\; A-C-E \not \in \Gamma_P.
\end{align*}
\end{lemma}

As in Section \ref{S:MainTheorem}, let
\begin{align*}
\Delta_P(B,C) &= \{ B+iP : i \in \ZZ \} \cap \Delta_P(C), \\
              &= \{ B+iP \in \Gamma_P, B-C+iP \not \in \Gamma_P \}.
\end{align*}
Furthermore, let
\begin{align*}
I_P(B,C) &= \{ i \in \ZZ : B+iP \in \Delta_P(C) \}, \\
         &= \{ i \in \ZZ : B+iP \in \Gamma_P, B-C+iP \not \in \Gamma_P \}. \\
I^\ast_P(B,C) &= \{ i \in \ZZ : B-C+iP \in \Delta_P(-C) \}, \\
         &= \{ i \in \ZZ : B-C+iP \in \Gamma_P, B+iP \not \in \Gamma_P \}.
\end{align*}
So that $I^\ast_P(B,C) = I_P(B-C,-C),$ 
and $\# I_P(B,C) - \# I^\ast_P(B,C) = \deg C$ (Lemma \ref{L:DminD}).
We rephrase some of the previous relations. 

\begin{lemma} 
\begin{align*}
&I_P(B,C) = I_P(B,C+E) \cap I_P(B+E,C+E). \\
&I^\ast_P(B,C+E) = I^\ast_P(B-E,C) \cap I^\ast_P(B,C). 
\end{align*}
\end{lemma}

\begin{lemma} 
\begin{align*}
&i \in I_P(B,C+E) \backslash  I_P(B+E,C+E) ~\Leftrightarrow~ i \in I_P(B-C,E) \;\wedge\; B+iP \in \Gamma_P, \\
&i \in I_P(B,C+E) \backslash I_P(B-E,C+E) ~\Leftrightarrow~ i \in I_P(B,E) \;\wedge\; B-C+E+iP \not \in \Gamma_P.
\end{align*}
\end{lemma}

\begin{proposition} 
\begin{align*}
&I_P(B,C+E) \backslash I_P(B,C) \cup I^\ast(B,C) \backslash I^\ast(B,C+E) = I_P(B-C,E). \\
&I_P(B,C+E) \backslash I_P(B-E,C) \cup I^\ast(B-E,C) \backslash I^\ast_P(B,C+E) = I_P(B,E).
\end{align*}
\end{proposition}

We describe the first partition for the following choice of divisors. For divisors $B_0$ and $C_0$ of degree zero, and for a point $Q$, 
let $B=B_0, C=C_0-2gQ, E=4gQ.$ Then
\[
I_P(B_0,C_0+2gQ) \cup I_P(B_0-C_0+2gQ,-C_0+2gQ) = I_P(B_0-C_0+2gQ,4gQ).
\]
In general, 
\[
\{ 0, \ldots, 2g-1 \} \subset I_P(B_0-C_0+2gQ,4gQ) \subset \{ -2g, \ldots, 4g-1 \}.
\]
The first inclusion follows with the definition of $I_P(B,C)$. For the second inclusion,
Lemma \ref{L:Ddegree} gives $0 \leq i+2g \leq 6g-1.$ 
 
\begin{proposition} \label{P:6g}
Let $B_0$ and $C_0$ be divisor classes of degree zero. Define partitions
$\{ -2g, \ldots, -1 \} = N_1 \cup G_1,$ $\{ 0, \ldots,2g-1 \} = N_2 \cup G_2,$ and
$\{ 2g, \ldots, 4g-1 \} = N_3 \cup G_3,$ such that
\begin{align*}
&k \in N_1 ~\Leftrightarrow~ B_0-C_0+2gQ+kP \in \Gamma_P, \\
&k \in N_2 ~\Leftrightarrow~ B_0+kP \in \Gamma_P, \\
&k \in N_3 ~\Leftrightarrow~ B_0-C_0-2gQ+kP \in \Gamma_P.
\end{align*}
Then $\# N_i = \# G_i = g$, for $i=1,2,3.$ Moreover
\[
I_P(B_0,C_0+2gQ) = N_2 \cup G_3 ~~\text{and}~~ I_P(B_0-C_0+2gQ,-C_0+2gQ) = N_1 \cup G_2.
\]
\end{proposition}
\begin{proof}
\[
\begin{array}{cccc}
                                      &\{-2g, \ldots, -1\}  &\{0, \ldots, 2g-1\}  &\{2g, \ldots, 4g-1 \} \\[1ex]
\begin{array}{rl} \{ k :&B_0 + kP \in \Gamma_P \;\wedge \\        
                        &B_0-C_0-2gQ+kP \in \Gamma_P \}  \end{array}        &-     &-      &N_3   \\[1.5ex]
\begin{array}{rl} \{ k :&B_0 + kP \in \Gamma_P  \;\wedge \\       
                        &B_0-C_0-2gQ+kP \not \in \Gamma_P \} \end{array}      &-     &N_2    &G_3   \\[1.5ex]
\begin{array}{rl} \{ k :&B_0 + kP \not \in \Gamma_P \;\wedge \\  
                        &B_0-C_0+2gQ+kP \in \Gamma_P \} \end{array}    &N_1   &G_2    &-     \\[1.5ex]
\begin{array}{rl} \{ k :&B_0 + kP \not \in \Gamma_P \wedge \\  
                        &B_0-C_0+2gQ+kP \not \in \Gamma_P \} \end{array} &G_1   &-      &-     
\end{array}
\]
\end{proof}

\section{Discrepancies} \label{S:Discrepancies}

We continue the description of a delta set $\Delta_P(C)$ in terms of other known delta sets. The results in the
previous section show that differences between similar delta sets, such as $\Delta_P(C+E)$ and $\Delta_P(C)$, for
$E \in \Gamma_P$, can be described in terms of the delta set $\Delta_P(E).$ In this section, we refine the results
for the special case that $E = Q$ is a point different from $P$.  
 
\begin{lemma}
For distinct points $P$ and $Q$, $\Delta_P(Q)=\Delta_Q(P)$.  
\end{lemma}
\begin{proof}
\begin{multline*}
A \in \Delta_P(Q) ~\Leftrightarrow~ L(A) \neq L(A-P) \;\wedge\; L(A-Q) = L(A-Q-P) \\
~\Leftrightarrow~ L(A) \neq L(A-Q) \;\wedge\; L(A-P) = L(A-P-Q) ~\Leftrightarrow~ A \in \Delta_Q(P).
\end{multline*}
\end{proof}

Let $D(P,Q) = \Delta_P(Q) = \Delta_Q(P).$ We call a divisor $A \in D(P,Q)$ a discrepancy for the points $P$ and $Q$.

\begin{lemma} \label{L:discdeg}
A divisor $A \in D(P,Q)$ is of degree $0 \leq \deg A \leq 2g$. The cases $\deg A = 0$ and $\deg A = 2g$
correspond to unique divisor classes $A=0$ and $A=K+P+Q$, respectively. Furthermore,
\[
A \in D(P,Q) ~\Leftrightarrow~ K+P+Q-A \in D(P,Q).
\]
\end{lemma}
\begin{proof} 
Use Lemma \ref{L:Ddegree}.
\end{proof} 

The set $\Delta_P(B,Q)$ is defined as $\{ B+kP : k \in \ZZ \} \cap \Delta_P(Q)$. 
It follows from Lemma \ref{L:DminD} that $\Delta_P(B,Q)$ is a singleton set. For a divisor
$B$, and for a given choice of distinct points $P$ and $Q$, define
\[
B_Q = \Delta_P(B,Q), \qquad B_P = \Delta_Q(B,P).
\]

\begin{lemma} \label{L:AisAp}
\[
A \in D(P,Q) ~\Leftrightarrow~ A = A_Q ~\Leftrightarrow~ A = A_P. 
\]
\end{lemma}
\begin{proof}
Clear after writing $A_Q = \Delta_P(A,Q)$ and $A_P = \Delta_Q(A,P).$
\end{proof}

\begin{lemma} \label{L:singleton}
For distinct points $P$ and $Q$, and for a divisor $B$, 
$B_Q = \Delta_P(B,Q) = B+kP,$  
for $k$ minimal such that $L(B+kP) \neq L(B+kP-Q)$. For a general
$k$, $B+kP \in \Gamma_Q$ if and only if $B+kP \geq B_Q.$
\end{lemma}
\begin{proof}
\begin{align*}
B+kP \in \Delta_P(Q) &~\Leftrightarrow~ B+kP \in \Delta_Q(P) \\
                     &~\Leftrightarrow~ B+kP \in \Gamma_Q \;\wedge\; B+(k-1)P \not \in \Gamma_Q.
\end{align*}
\end{proof}
For distinct points $P$ and $Q$, and for a divisor $B$, let 
\[
D_B(P,Q) = \{ B+iP+jQ : i, j \in \ZZ \} \cap D(P,Q),
\]
and define functions $\sigma = \sigma_B, \tau = \tau_B : \ZZ \longrightarrow \ZZ$ such that
\begin{align*}
&B+iP+jQ \in \Gamma_P ~\Leftrightarrow~ j \geq \sigma(i), \\
&B+iP+jQ \in \Gamma_Q ~\Leftrightarrow~ i \geq \tau(j).
\end{align*}
With the lemma,
\begin{align*}
&(B+iP)_P = \Delta_Q(B+iP,P) = B+iP+\sigma(i)Q. \\
&(B+jQ)_Q = \Delta_P(B+jQ,Q) = B+\tau(j)P+jQ.
\end{align*}

\begin{theorem}
For a divisor $B$, 
\[
D_B(P,Q) = \{ B+iP+\sigma(i)Q : i \in \ZZ \} = \{ B+\tau(j)P+jQ : j \in \ZZ \}.
\]
In particular, the functions $\sigma = \sigma_B$ and $\tau = \tau_B$ are mutual inverses and describe permutations
of the integers. For a divisor $B$ of degree zero, and for $i \in \ZZ$,
$-i \leq \sigma(i), \tau(i) \leq 2g-i.$ For $m$ such that $mP \sim mQ$, the functions $i+\sigma(i), j+\tau(j)$ only depend
on $i,j$ modulo $m$. The functions $\sigma, \tau$ are determined by their images 
on a full set of representatives for $\ZZ /m \ZZ.$
\end{theorem}

\begin{proof}
For the second claim use Lemma \ref{L:discdeg}. Finally, $B+iP+jQ \in D(P,Q)$ only
depends on the divisor class of $B+iP+jQ$ and therefore
\[ 
B+iP+jQ \in D(P,Q) ~\Leftrightarrow~ B+(i+m)P+(j-m)Q \in D(P,Q), 
\]
so that $\sigma(i+m) = \sigma(i)-m.$
\end{proof}

The discrepancies $D_B(P,Q)$ serve as an index set for a common basis of the vector spaces $L(B+aP+bQ)$,
for $a, b \in \ZZ.$

\begin{theorem} \label{T:basis}
\[
\dim L(B+aP+bQ) = \# \{ B+iP+jQ \in D(P,Q) : i \leq a \wedge j \leq b \}.
\]
\end{theorem} 

\begin{proof}
$\dim L(B+aP+bQ) \neq \dim L(B+aP+bQ-P)$ if and only if $B+aP+bQ \in \Gamma_P$ if and only
if $B+aP+bQ \geq (B+aP)_P \in D_B(P,Q)$ if and only if there exists 
$B+iP+jP \in D_B(P,Q)$ with $i=a, j \leq b.$ Use induction on $a$ to complete the proof.
\end{proof}

\begin{theorem} \label{T:ApA} For given distinct points $P$ and $Q$, and for divisors $A$ and $C$,
\[
A \in \Delta_P(C+Q) ~\Leftrightarrow~ A_P \leq A \leq (A-C)_P + C.
\]
Moreover,
\begin{align*}
A_P = A \leq (A-C)_P + C  &~\Leftrightarrow~ A \in \Delta_P(C+Q) \;\wedge\; A-Q \not \in \Delta_P(C+Q). \\
A_P \leq A = (A-C)_P + C  &~\Leftrightarrow~ A \in \Delta_P(C+Q) \;\wedge\; A+Q \not \in \Delta_P(C+Q), \\
                          &~\Leftrightarrow~ A \in \Delta_P(C+Q) \;\wedge\; A \not \in \Delta_P(C). 
\end{align*}
\end{theorem}

\begin{proof} With Lemma \ref{L:singleton}, $A \in \Gamma_P$ if and only if $A \geq A_P,$ and
$A-C-Q \not \in \Gamma_P$ if and only if $A-C \leq (A-C)_P.$ The last claims use Lemma \ref{L:diffB} (part two)
and Lemma \ref{L:diffA} (part one), respectively.
\end{proof}

Note that the divisors $A_P,$ $A$, and $(A-C)_P+C$, have the same multiplicities at any point other than $Q$.
Let $B_0$ and $C_0$ be divisors of degree zero, and let $\sigma = \sigma_{B_0}$ and $\sigma' = \sigma_{B_0-C_0}$.

\begin{corollary}
\begin{align*}
&B_0+kP+\ell Q \in \Delta_P(C_0+iP+jQ+Q). \\
&\quad ~\Leftrightarrow ~k+\sigma(k) \;\leq\; k+\ell \;\leq\; (k-i)+\sigma'(k-i) + i+j. \\
&\quad ~\Leftrightarrow ~\sigma(k) \;\leq\; \ell \;\leq\; \sigma'(k-i)+j. \\
\end{align*}
\end{corollary}
\begin{proof} For $A = B_0+kP+\ell Q,$ $A_P = B_0+kP+\sigma(k) Q$. For $C=C_0+iP+jQ$, $(A-C)_P = B_0-C_0+(k-i)P+\sigma'(k-i)Q.$
Now use the theorem, and compare either the degrees of the divisors $A_P$, $A$, and $(A-C)_P+C$, or their multiplicities at
$Q$. 
\end{proof}

Let $\tau$ and $\tau'$ be the inverse functions for $\sigma=\sigma_{B_0}$ and $\sigma'=\sigma_{B_0-C_0}$, respectively.
Let
\begin{align*}
&d_P(k) = \deg (B_0+kP)_P = k+\sigma(k), \\
&d_Q(\ell) = \deg (B_0+\ell Q)_Q = \tau(\ell)+\ell. \\[1ex]
&d'_P(k-i) = \deg (B_0-C_0+(k-i)P)_P = k-i+\sigma'(k-i), \\
&d'_Q(\ell-j) = \deg (B_0-C_0+(\ell-j)Q)_Q = \tau'(\ell-j)+\ell-j. 
\end{align*}
For $mP \sim mQ,$ the functions $d_P, d_Q$ and $d'_P, d'_Q$ are defined modulo $m$.

\begin{proposition} \label{P:tables}
Let $A=B_0+kP+\ell Q$, and let $C=C_0+iP+jQ+Q$.
\begin{align*} 
&A_P = A \leq (A-C)_P + C \\
&\quad ~\Leftrightarrow~  k = k^+ = \tau(\ell) \;\wedge\; \ell-j \leq \sigma'(k^+-i). \\
&\quad ~\Leftrightarrow~  k = k^+ = d_Q(\ell) - \ell \;\wedge\; d_Q(\ell) \leq d'_P(k^+-i) + i+j. \\[1ex]
&A_P \leq A = (A-C)_P + C \\
&\quad ~\Leftrightarrow~  k = k^- = \tau'(\ell-j)+i \;\wedge\; \sigma(k^-) \leq \ell. \\
&\quad ~\Leftrightarrow~  k = k^- = d'_Q(\ell-j)-\ell+i+j \;\wedge\; d_P(k^-) \leq d'_Q(\ell-j)+i+j.                               
\end{align*}
\end{proposition}

We use the proposition to create tables for each of the three equivalences in Theorem \ref{T:ApA}. 
The tables $N$ and $K$ are used to
compute the size of a delta set or to construct a delta set, respectively (Example \ref{E:herm}). 
The tables $N^+$ and $N^-$ are used in the optimization of the order bound (Example \ref{E:suz1}). 
The tables $K^+$ and $K^-$ provide more information that can be used for further improvements 
with the ABZ bound (Example \ref{E:suz2}). In all cases, let $A=B_0+kP+\ell Q$, and let $C=C_0+iP+jQ+Q$.
\[
A_P \leq A = (A-C)_P + C  ~\Leftrightarrow~ A \in \Delta_P(C+Q) \;\wedge\; A \not \in \Delta_P(C).
\]
\[
K_\ell(i,j) = \tau'(\ell-j)+i , \qquad N_\ell(i,j) = \begin{cases} 1 &~~\text{if $\sigma(k) \leq \ell$} \\ 0 &~~\text{if $\sigma(k) > \ell$} \end{cases}
\]
The table $N_\ell(i,j)$ indicates whether there exists $A \in \Delta_P(B_0+\ell Q,C_0+iP+jQ+Q) \backslash \Delta_P(B_0+\ell Q,C_0+iP+jQ)$.
In the affirmative case (N=1), the table $K_\ell(i,j)$ gives the unique value of $k$ such that $A=B_0+kP+\ell Q.$ The table $N$ is
sufficient for computing the size of a delta set, the table $K$ moreover provides the elements of a delta set. 
\[
A_P = A \leq (A-C)_P + C  ~\Leftrightarrow~ A \in \Delta_P(C+Q) \;\wedge\; A-Q \not \in \Delta_P(C+Q). 
\]
\[
K^+_i(j,\ell) = \tau(\ell) , \qquad N^+_i(j,\ell) = \begin{cases} 1 &~~\text{if $\sigma'(k-i) \geq \ell-j$} \\ 0 &~~\text{if $\sigma'(k-i) < \ell-j$} \end{cases}
\]
The table $N^+_i(j,\ell)$ indicates whether there exists $A \in \Delta_P(B_0+\ell Q,C_0+iP+jQ+Q)$ with 
$A-Q \not \in \Delta_P(B_0+\ell Q-Q,C_0+iP+jQ)$.
In the affirmative case (N=1), the table $K^+_i(j,\ell)$ gives the unique value of $k$ such that $A=B_0+kP+\ell Q.$ The table $N^+$ is
sufficient for comparing the sizes of two delta sets, the table $K^+$ moreover provides the elements for the difference in one
direction. 
\[
A_P \leq A = (A-C)_P + C  ~\Leftrightarrow~ A \in \Delta_P(C+Q) \;\wedge\; A+Q \not \in \Delta_P(C+Q).
\] 
\[
K^-_i(j,\ell) = \tau'(\ell-j)+i , \qquad N^-_i(j,\ell) = \begin{cases} 1 &~~\text{if $\sigma(k) \leq \ell$} \\ 0 &~~\text{if $\sigma(k) > \ell$} \end{cases}
\]
The table $N^-_i(j,\ell)$ indicates whether there exists $A \in \Delta_P(B_0+\ell Q,C_0+iP+jQ+Q)$ with $A+Q \neq \Delta_P(B_0+\ell Q+Q,C_0+iP+jQ+Q)$. 
In the affirmative case (N=1), the table $K^-_i(j,\ell)$ gives the unique value of $k$ such that $A=B_0+kP+\ell Q.$ 
The table $N$ is sufficient for comparing the sizes of two delta sets, the table $K$ moreover provides the elements in the difference. 

\section{Hermitian curves} \label{S:Hermitian}
 
Let $X$ be the Hermitian curve over $\mathbb{F}_{q^2}$ defined by the equation $y^q+y=x^{q+1}$. 
The curve has $q^3+1$ rational points and genus $g = q(q-1)/2$. Let $P$ and $Q$ be two distinct
rational points. We will give a description of the set 
\[
D_0(P,Q) = D(P,Q) \cap \{ iP+jQ : i,j \in \ZZ \}.
\]
We use this description to determine lower bounds for $\gamma_P(C),$ 
for $C \in \{ iP+jQ : i,j \in \ZZ \}.$ 
The only property of the two rational points that we use is that lines intersect the pair $(P,Q)$
with one of the multiplicities
\[
\begin{array}{rrr} (0,0) &(0,1) &(0,q+1) \\  (1,0) &\underline{(1,1)} \\ (q+1,0) \end{array}
\]
The curve is a smooth plane curve and if $H$ is the intersection divisor of a line then $K=(q-2)H$ 
represents the canonical class. We have $H \sim (q+1)P \sim (q+1)Q$ and $m(P-Q)$ is principal for
$m=q+1.$

\begin{proposition} \label{P:discherm}
The $m$ inequivalent divisor classes in $D_0(P,Q)$ are represented by the divisors
\[
dH - d P - d Q,  \quad \text{for $d=0,1,\ldots,q.$}
\]
\end{proposition}
\begin{proof}
Since $m=q+1$ is minimal such that $mP \sim mQ$, the divisors are inequivalent.
As multiples of $H-P-Q \in \Gamma_P$, each of the divisors $dH - d P - d Q \in \Gamma_P,$
for $d=0,1,\ldots,q.$ A divisor $A \in \Gamma_P$ is a discrepancy if and only if $K+P+Q-A \in \Gamma_P.$
Now use $K+P+Q=(q-2)H+P+Q=q(H-P-Q).$ 
\end{proof} 
 
The function $y$ has divisor $y=(q+1)(P_0-P_\infty)$, where $P_0 = (0,0)$ and $P_\infty=(0:1:0)$.
Moreover $L(H-P_0-P_\infty) = \langle x \rangle.$ 

\begin{corollary}
The ring $O$ of functions that are regular outside $P_0$ and $P_\infty$ has a basis
$\langle\; x^i y^j \,|\, 0 \leq i \leq q, j \in \ZZ \;\rangle$ as vector space over $\mathbb{F}_{q^2}$.
\end{corollary}
\begin{proof}
Theorem \ref{T:basis}.
\end{proof} 

\begin{lemma} \label{L:hermP}
Let $C = dH-aP-bQ,$ for $0 \leq a \leq q, 0 \leq b \leq q+1.$ 
\[
C \in \Gamma_P ~\Leftrightarrow~ d > a ~\text{or}~ d = a \geq b. 
\]
\end{lemma}

\begin{proof} 
We use Lemma \ref{L:singleton} together with Proposition \ref{P:discherm}.
We may assume $H=(q+1)Q.$ Then, $dH-aP-bQ \in \Gamma_P$ if and only if 
$dH-aP-bQ \geq aH-aP-aQ$ if and only if $d>a$ or $d = a \geq b$.
\end{proof}
         
\begin{proposition} \label{P:short}
Let $C = dH-aP-bQ,$ for $0 \leq a, b \leq q.$ The set
$\Delta_P(-C) = \{ A \in \Gamma_P : A+C \not \in \Gamma_P \}$ contains
the following elements
\begin{align*}
&(q-1-d-r)H-(q-a)P, &\text{for $d \leq d+r < a.$} \\
&sH, &\text{for $d \leq d+s < a.$} \\
&sH-(d+s-a)P &\text{for $d \leq a \leq d+s < b.$} 
\end{align*}
\end{proposition}

\begin{proof}
With the lemma, $A=(q-1-d-r)H-(q-a)P \in \Gamma_P$ for $q-a \leq q-1-d-r,$ and $A+C = (q-1-r)H-qP-bQ \not \in \Gamma_P$
for $r \geq 0.$ Clearly, $A=sH \in \Gamma_P$ for $s \geq 0$, and $A+C = (d+s)H-aP-bQ \not \in \Gamma_P$ for $a > d+s.$
Finally, $A=sH-(d+s-a)P \in \Gamma_P$ for $0 \leq d+s-a \leq s$, and $A+C = (d+s)H-(d+s)P-bQ \not \in \Gamma_P$ for
$b > d+s.$ 
\end{proof} 

\begin{lemma}
Let $0 \leq a, b \leq q.$ There exists a form of degree $d$ that intersects the curve
in $(P,Q)$ with precise multiplicities $(a,b)$ if and only if $0 \leq a, b \leq d.$
\end{lemma}
\begin{proof}
Such a curve exists if and only if $dH-aP-bQ \in \Gamma_P \cap \Gamma_Q.$ With Lemma
\ref{L:hermP}, the latter holds if and only $0 \leq a, b, \leq d.$
\end{proof}

\begin{lemma} \label{L:case1}
For $d \geq 0$, let $C$ be a divisor with $dH-dP-dQ \leq C \leq dH$. Then $C$ has no
base points and $\gamma_P(C) = \deg C.$
\end{lemma}
\begin{proof}
Since $C$ is equivalent to an effective divisor with support in $P$ and $Q$,
those two points are the only candidates for the base points. With Lemma \ref{L:hermP},
$C \in \Gamma_P \cap \Gamma_Q$, and therefore neither $P$ nor $Q$ is a base point.
The last claim uses Proposition \ref{P:gammap}.
\end{proof}

\begin{proposition}
Let 
\begin{align*}
&C = dH -aP -bQ, \qquad \text{for $d \in \ZZ, 0 \leq a,b \leq q,$} \\
&A = jH+i(H-P), \qquad \text{for $j \in \ZZ, 0 \leq i \leq q.$}
\end{align*}
Then $A \in \Delta_P(C)$ if and only if
\[
\begin{cases}
0 \leq j \leq (d-a+q-1), &\text{if $0 \leq i < a-b$} \\
0 \leq j \leq (d-a+q-2), &\text{if $a-b \leq i < a$} \\
0 \leq j \leq (d-a-1), &\text{if $a \leq i < a-b+q+1$} \\
0 \leq j \leq (d-a-2), &\text{if $a-b+q+1 \leq i \leq q$} 
\end{cases}
\]
\end{proposition}

\begin{proof}
For $A-C$ we write 
\[
\begin{cases} &(j+i-d+1)H-(i-a)P-(q+1-b)Q, ~~\text{if $i-a \geq 0.$} \\
                    &(j+i-d+2)H-(i-a+q+1)P-(q+1-b)Q, ~~\text{if $i-a < 0.$} 
      \end{cases}
\]
With Lemma \ref{L:hermP}, $A-C \not \in \Gamma_P$ if and only if
\[
\begin{cases}
&j < d-a-1, ~~\text{or}~~ j=d-a-1, i-a < q+1-b, ~~\text{if $i-a \geq 0.$} \\
&j < d-a-1+q ~~\text{or}~~ j=d-a-1+q, i-a < -b, ~~\text{if $i-a < 0.$}
\end{cases}
\]
In combination with $A \in \Gamma_P$ if and only if $j \geq 0$, this proves the claim. 
\end{proof}

\begin{corollary} \label{C:delt}
Let $C=dH-aP-bQ,$ for $0 \leq a, b \leq q.$ \\
For $a-d < 0$, 
\[
\Delta_P(-C) = \emptyset, \qquad \# \Delta_P(0,C) = \deg C.
\]
For $0 \leq a-d \leq q-1$,
\begin{align*}
&\# \Delta_P(0,C) = a(q-1-a+d)+\max\{0,a-b\}. \\ 
&\# \Delta_P(0,-C) = (q+1-a)(a-d)+\max\{0,b-a\}. 
\end{align*}
For $a-d > q-1$, 
\[
\Delta_P(C) = \emptyset, \qquad \# \Delta_P(0,-C) = - \deg C.
\] 
\end{corollary}

\begin{theorem}\label{mult2}
Let $C = dH -aP -bQ,$ for $d \in \ZZ,$ and for $0 \leq a,b \leq q.$ Then
\[
\begin{array}{lcl}
(\text{\rm{Case 1 : }} a, b \leq d) & &\gamma_{P}(C) \;=\; \gamma_Q(C) \;=\; \deg C. \\
(\text{\rm{Case 2a : }} b \leq d \leq a) & &\gamma_{P}(C) \;\geq\; \deg C + a-d. \\ 
(\text{\rm{Case 2b : }} a \leq d \leq b) & &\gamma_{Q}(C) \;\geq\; \deg C + b-d. \\ 
(\text{\rm{Case 3a : }} d \leq a \leq b, a<q) & &\gamma_{P}(C) \;\geq\; \deg C + a-d+b-d. \\
(\text{\rm{Case 3b : }} d \leq b \leq a, b<q) & &\gamma_{Q}(C) \;\geq\; \deg C + a-d+b-d. \\ 
(\text{\rm{Case 4 : }} d \leq a=b=q) & &\gamma_{P}(C) \;=\; \gamma_Q(C) \;\geq\; \deg C + q - d.
\end{array}
\]
\end{theorem}

\begin{proof} (Case 1) uses Lemma \ref{L:case1}. The lower bounds follow from Proposition \ref{P:short} by using 
$\gamma_{P}(C) = \deg C + \gamma_{P}(-C)$. Or we can obtain the lower bounds from Corollary \ref{C:delt}
in combination with $\gamma_P(C) \geq \# \Delta_P(0,C)$ (Corollary \ref{C:cbfrac}). For $0 \leq a-d \leq q-1$, 
\[
\# \Delta_{P_\infty}(0,C) = a(q-1-a+d)+\max\{0,a-b\}.
\]
(Case 2a: $b \leq d \leq a$) $~a(q-1-a+d)+a-b - d(q-1)+b-d = (a-d)(q-a) \geq 0.$ \\
(Case 3a : $d \leq a \leq b$) $~a(q-1-a+d)+0 - d(q-1) = (a-d)(q-1-a) \geq 0.$ \\
(Case 4 : $d \leq a=b=q$) $~q(d-1) = d(q+1)-q-q+(q-d).$
\end{proof}

\begin{theorem}\label{thm:dist}
For $G=K+C$, and for $D \cap S=\emptyset$, the algebraic geometric code $C_\Omega(D,G)$
has minimum distance $d \geq \gamma(C;S)$.
Let $C = dH -aP -bQ,$ for $d \in \ZZ,$ and for $0 \leq a,b \leq q.$ Then
\[
\begin{array}{lcl}
(\text{\rm{Case 1 : }} a, b \leq d) & &\gamma(C) \;\geq\; \deg C. \\
(\text{\rm{Case 2a : }} b \leq d \leq a) & &\gamma(C;P) \;\geq\; \deg C + a-d. \\ 
(\text{\rm{Case 2b : }} a \leq d \leq b) & &\gamma(C;Q) \;\geq\; \deg C + b-d. \\ 
(\text{\rm{Case 3a : }} d \leq a \leq b, a<q) & &\gamma(C;P,Q) \;\geq\; \deg C + a-d+b-d. \\
(\text{\rm{Case 3b : }} d \leq b \leq a, b<q) & &\gamma(C;P,Q) \;\geq\; \deg C + a-d+b-d. \\ 
(\text{\rm{Case 4 : }} d \leq a=b=q) & &\gamma(C;P) = \gamma(C;Q) \;\geq\; \deg C + q - d.
\end{array}
\]
\end{theorem}
\begin{proof}
Use the order bound with \\[1ex]
(Case 2a: $b \leq d \leq a$) $Q_0 = \ldots = Q_{a-d-1} = P.$ \\
(Case 3a : $d \leq a \leq b$) $Q_0 = \ldots = Q_{a-d-1} = P,$ $Q_{a-d} = \ldots = Q_{a-d+b-d-1} = Q.$ \\
(Case 4 : $d \leq a=b=q$) $Q_0 = \ldots = Q_{q-d-1} = P.$
\end{proof}

The following example illustrates the use of the tables $K_\ell(i,j)$ and $N_\ell(i,j)$ for constructing a delta set $\Delta_P(\ell Q,iP+jQ+Q)$
(Proposition \ref{P:tables}). In this case, the functions $d_P=d_Q=d'_P=d'_Q$ all agree and we can omit the index. 

\begin{example} \label{E:herm}
For the Hermitian curve of degree four, the genus $g=3.$ The discrepancies $D_0(P,Q)$ 
are represented by the divisors $0, H-P-Q, 2H-2P-2Q, 3H-3P-3Q.$ In particular, $d(k) = 0, 2, 4, 6,$ for $k = 0, -1, -2, -3$ modulo $4$,
respectively. 
\[
\begin{array}{lrrrrrrrrrrrrrr}
(\ell=0,i=1)   &j               &= -7    &-6   &-5   &-4   &-3 &-2  &-1 &0 &1 &2 &3  &4   \\[1ex] \hline
\multicolumn{2}{r}{d(\ell-j)}                     &2     &4    &6    &0    &2  &4   &6  &0   &2  &4  &6  &0  \\ 
(K) &k                     &(-4)  &(-1) &(2)  &(-3) &0  &3   &6  &(1) &4  &7  &10 &(5)  \\[1ex]
&d(k)                      &0     &2    &4    &6    &0  &2   &4  &6   &0  &2  &4  &6 \\
\multicolumn{2}{r}{d(k)-d(\ell-j)}               &-2    &-2   &-2   &6    &-2 &-2  &-2 &6   &-2 &-2 &-2 &6 \\
(N) &\leq i+j                    &0     &0    &0    &0    &1  &1   &1  &0   &1  &1  &1  &0 \\[1ex] \hline
\end{array} \\[1ex]
\]
The value for $k = d(\ell-j)-\ell+i+j.$
Row (K) gives the difference $\Delta_P(\ell Q,iP+jQ+Q) \backslash \Delta_P(\ell Q, iP+jQ) = \{ kP+\ell Q \} \cap \Gamma_P$, with empty intersection
if and only if $k$ appears in parentheses. Row (N) has the decision whether the difference is empty (N=0) or nonempty (N=1).
As a special case, we see that $\Delta_P(0,P+2Q) = \{ 0, 3P, 6P, 4P \}$. 
The numbers in parentheses illustrate the duality in Proposition \ref{P:6g}.
\begin{align*}
&\Delta_P(0,P+5Q) = 0 + \{ 0, 3P, 6P, 4P, 7P, 10P \}, \\
&\Delta_P(-P+7Q,-P+7Q) = (-P+7Q) + \{ -4P, -P, 2P, -3P, P, 5P \}. 
\end{align*} 
With $-P+7Q \sim 7P-Q,$ 
\[
\Delta_P(-Q,-P+7Q) = \{ 3P-Q, 6P-Q, 9P-Q, 4P-Q, 8P-Q, 9P-Q \}.
\] 
The partition of the interval $\{ -2g, \ldots, 4g-1 \}$ in Proposition \ref{P:6g} is given by 
\[
\begin{array}{lll}
                         &                      &N_3 = \{ 8, 9, 11 \} \\
                         &N_2 = \{ 0, 3, 4 \}   &G_3 = \{ 6, 7, 10 \} \\
N_1 = \{ -4, -3, -1 \}   &G_2 = \{ 1, 2, 5 \}   & \\
G_1 = \{ -6, -5, -2 \}   & & \\
\end{array}
\]
\end{example}

\section{Suzuki curves} \label{S:Suzuki}

The Suzuki curve over the field of $q = 2q_0^2$ elements is defined by the equation $y^q+y=x^q_0(x^q+x).$
The curve has $q^2+1$ rational points and genus $g = q_0(q-1)$. 
The semigroup of Weierstrass nongaps at a rational point is
generated by $\{q, q+q_0, q+2q_0, q+2q_0+1\}$. For any two rational points $P$ and $Q$
there exists a function with divisor $(q+2q_0+1)(P-Q).$ Let $m = q+2q_0+1=(q_0+1)^2+{q_0}^2$,
and let $H$ be the divisor class containing $mP \sim mQ$. The divisor $H$ is very ample and
gives an embedding of the Suzuki curve in ${\mathbb P}^4$ as a smooth curve of degree $m$. 
The canonical divisor $K \sim 2(q_0-1)H.$
A hyperplane $H$ intersects $(P,Q)$ with one of the following multiplicities.
\[
\begin{array}{rrrrr} 
(0,0)        &(0,1)      &(0,q_0+1) &(0,2q_0+1) &(0,q+2q_0+1) \\  
(1,0)        &(1,1)      &(1,q_0+1) &\underline{(1,2q_0+1)} \\ 
(q_0+1,0)    &(q_0+1,1)  &\underline{(q_0+1,q_0+1)} \\
(2q_0+1,0)   &\underline{(2q_0+1,1)} \\
(q+2q_0+1,0) 
\end{array}
\]
Let
\[
D_0 = H-(2q_0+1)P-Q, ~
D_1 = H-(q_0+1)(P+Q),  ~
D_2 = H-P-(2q_0+1)Q.
\]
Then $L(D_i) \neq 0,$ and $L(D_i-P) = \dim L(D_i-Q) = 0,$ for $i=0,1,2.$
And $D_i \in D(P,Q)$, for $i=0,1,2.$ 

\begin{lemma} \label{L:q0}
For any nonnegative integer $q_0$,
\[
\{ -q_0(q_0+1), \ldots, +q_0(q_0+1) \} = \{ a(q_0+1)+bq_0 : |a| + |b| \leq q_0 \}.
\]
\end{lemma}

\begin{theorem} \label{T:descsuz}
The $m$ inequivalent divisor classes in $D_0(P,Q)$ are represented by
\begin{align*}
&i D_0 +j D_2, \quad \text{for $0 \leq i, j \leq q_0,$ and}  \\
&D_1 + i' D_0 +j' D_2, \quad \text{for $0 \leq i', j' \leq q_0-1.$}
\end{align*}
The given representatives correspond one-to-one to the $m$ divisors 
\[
D(a,b) = (a+q_0)H\,-\,((a+q_0)(q_0+1)+bq_0)P\,-\,((a+q_0)(q_0+1)-bq_0)Q,
\]
for $|a|+|b| \leq q_0.$ 
\end{theorem} 
\begin{proof}
\begin{align*}
&i D_0 + j D_2 \\
=~&i(H-(2q_0+1)P-Q)\,+\,j(H-P-(2q_0+1)Q), \\
=~&(i+j)H-(i+j)(q_0+1)(P+Q)-(i-j)q_0(P-Q). 
\end{align*}
Moreover, $0 \leq i,j \leq q_0$ if and only if 
$|i+j-q_0|+|i-j| \leq q_0.$ Thus 
\[
\{ i D_0 + j D_2 :  0 \leq i, j \leq q_0 \} = \{ D(a,b) : |a|+|b| \leq q_0,~ a - b \equiv 0 \pmod{2} \}
\]
\begin{align*}
&H-(q_0+1)(P+Q) + i' D_0 + j' D_2 \\
=~&(i'+j'+1)H-(i'+j'+1)(q_0+1)(P+Q)-(i'-j')q_0(P-Q). 
\end{align*}
Similarly, $0 \leq i',j' \leq q_0-1$ if and only if 
$|i'+j'+1-q_0|+|i'-j'| \leq q_0-1.$ And
\begin{align*}
&\{ H-(q_0+1)(P+Q) + i' D_0 + j' D_2 :  0 \leq i', j' \leq q_0-1 \} \\
=~&\{ D(a,b) : |a|+|b| \leq q_0,~ a - b \equiv 1 \pmod{2} \}
\end{align*}
We have constructed $m$ inequivalent divisors in $\Gamma_P$. A divisor
$A \in \Gamma_P$ is a discrepancy if and only if $K+P+Q-A \in \Gamma_P.$
With $K=2(q_0-1)H$, we see that
\begin{align*}
D(a,b)+D(-a,-b) &= (2q_0)H-(2q_0(q_0+1)P -2q_0(q_0+1)Q \\
                &= (2q_0-2)H+P+Q = K+P+Q.
\end{align*}
\end{proof} 

As an illustration, we give the discrepancies for the Suzuki curve $y^8+y=x^{10}+x^3$
over the field of eight elements ($q_0= 2, q=8, g=14, N=65, m=13=3^2+2^2$). 
\[
\begin {array}{ccccc} 
0&\cdot&H-5P-Q&\cdot&2H-10P-2Q\\
\cdot&H-3P-3Q&\cdot&2H-8P-4Q&\cdot\\
H-P-5Q&\cdot&2H-6P-6Q&\cdot&3H-11P-7Q\\
\cdot&2H-4P-8Q&\cdot&3H-9P-9Q&\cdot\\
2H-2P-10Q&\cdot&3H-7P-11Q&\cdot&4H-12P-12Q
\end {array} 
\]
With $H \sim 13Q$, we obtain the following multiplicities for the discrepencies at $(P,Q)$. 
\[
\begin {array}{ccccc} 
(0,0)&\cdot&(-5,12)&\cdot&(-10,24)\\
\cdot&(-3,10)&\cdot&(-8,22)&\cdot\\
(-1,8)&\cdot&(-6,20)&\cdot&(-11,32)\\
\cdot&(-4,18)&\cdot&(-9,30)&\cdot\\
(-2,16)&\cdot&(-7,28)&\cdot&(-12,40)
\end {array} 
\]
For the given Suzuki curve, Beelen \cite{Bee07FF} gives
an example of a two-point code for which the floor bound exceeds the order bound. 
The example generalizes to any Suzuki curve. For both the Suzuki curve over $\ff_8$ and
over $\ff_{32}$ (for which $q_0=4, q=32, g=124, N=1025, m=41=5^2+4^2$), the example is 
the only two-point code for which the floor bound exceeds the order bound.

\begin{example}
Let $A = B = K-H, Z=2P+2Q.$ With $\dim L(H) - \dim L(H-2P-2Q) = 4$, it follows 
that $L(A+Z)=L(A)$ and $L(B+Z)=L(B).$ For the code $C_\Omega(D,G) = C_L(D,G)^\perp$ with
$G=K+C=A+B+Z=2K-2H+2P+2Q$, the threshold divisor $C=K-2H+Z$. The floor bound gives 
minimum distance $d \geq \deg C + \deg Z = d^\ast + 4.$ 
This is one better than the order bound.
\end{example}

We give an example of the ABZ bound for codes that improves both the floor bound and
the order bound.

\begin{example} \label{E:abz}
Let $A = B = K-H, Z=(q_0+2)P+2Q$. For the code $C_\Omega(D,G) = C_L(D,G)^\perp$ with
$G=K+C=A+B+Z=2K-2H+(q_0+2)P+2Q$, the threshold divisor $C=K-2H+Z$. For the
ABZ bound we use $\dim L(A) - \dim L(A-C) + \dim L(B) - \dim L(B-C) = 
2 (\dim L(K-H) - \dim L(H-Z)) = 2 \deg (K-H) - \deg K + 2 \dim L(H).$
The bound $d \geq 10$ for $q_0=2$ is one better than both the floor bound and
the order bound.
\end{example}

We illustrate the use of tables $K^\pm_i(j,\ell)$ and $N^\pm_i(j,\ell)$ for the comparison of delta sets $\Delta_P(\ell Q,iP+jQ+Q)$
and $\Delta_P(\ell Q \mp Q,iP+jQ+Q)$ (Proposition \ref{P:tables}). The functions $d_P=d_Q=d'_P=d'_Q$ all agree and we can omit the index. 
The functions are defined on residue classes modulo $m$. With Lemma \ref{L:q0} and Theorem \ref{T:descsuz}, for $|a|+|b| \leq q_0$, 
\[
d(k) = (q_0-a)(q-1), \quad \text{for $k=a(q_0+1)+bq_0-q_0(q_0+1) \pmod{m}$}
\]

\begin{example} \label{E:suz1}
For the Suzuki curve over $\ff_{32}$, let $C=55P+31Q.$ In this case there is a unique choice $B=-5Q$ such that
$\Delta_P(B,C) \geq 90.$ The improvement over the choice $B=0$ can be seen as follows.
\begin{align*}
&\begin{array}{rrrrrrrrrrrrr}
(i=55,j=30)   &\ell                   &=0   &-1  &-2  &-3   &-4    &-5    \\[1ex] \hline
&d(\ell)                   &0   &31  &62  &93   &124 &\cdot \\
(K^+)    &k^+   &0   &32  &64  &(96) &(128) &\cdot \\[1ex]
&d(k^+-i)                  &62  &93  &124 &0    &31 &\cdot \\
&d(\ell)-d(k^+-i)          &-62 &-62 &-62 &93   &93  &\cdot \\
(N^+) &\leq i+j                  &1   &1   &1   &0    &0  &\cdot  \\[1ex] \hline
\end{array} 
\end{align*}
\begin{align*}
&\begin{array}{rrrrrrrrrrrrr}
(i=55,j=30)   &\ell                   &= 0   &-1  &-2  &-3   &-4    &-5    \\[1ex] \hline
&d(\ell-j)     &\cdot &217 &124 &155 &186 &217 \\ 
(K^-) &k^-      &\cdot &303 &(211) &243 &275 &307 \\[1ex]
&d(k^-)                       &\cdot &155 &217 &93  &124 &155 \\
&d(k^-)-d(\ell-j)             &\cdot &-62  &93   &-62  &-62  &-62  \\
(N^-) &\leq i+j                &\cdot &1    &0    &1    &1    &1 \\[1ex] \hline
\end{array}
\end{align*}
The tables use $k^+ = d(\ell) - \ell$ and $k^- = d(\ell-j)-\ell+i+j.$ \\

From the tables we obtain
\begin{align*}
I_P(0,C) \backslash I_P(-5Q,C) &= \{ k^+ \in I_P(-5Q,5Q) : k^+ P \in \Gamma_P \} \\
                               &= \{ 0, 32, 64, (96), (128) \}. \\
I_P(-5Q,C) \backslash I_P(0,C) &= \{ k^- \in I_P(-C,5Q) : -5Q+k^- P \not \in \Gamma_P \} \\
                               &= \{ 307, 275, 243, (211), 303 \}.
\end{align*}
The net gain is therefore $4-3=1.$ To reach this conclusion it is sufficient to consult the 
rows ($N^+$) and ($N^-$).
\begin{align*}
\Delta_P(55P+31Q) \supseteq &\{ A_1 = 36P-5Q, \ldots, A_{45} = 163P-5Q \} \\
 & \cup \{ A_{46} = 180P-5Q, \ldots, A_{90} = 307P-5Q \}
\end{align*}
\end{example}

\begin{example} \label{E:suz2}
We illustrate the improvemnt of the ABZ bound for cosets over the order bound.
Both $\# \Delta_P(0,9P+9Q) = \# \Delta_P(9Q,9P+9Q) = 40.$ This is the optimum
for the order bound. For $r \geq 0$,
\[
\# \Delta_P(0,C) = \# \Delta_P(\leq rP,C) + \# \Delta_P(\geq rP+P,C). 
\]
For $r, s \geq 0$ such that
\[
\# \Delta_P(\geq rP+P+sQ,C) > \# \Delta_P(\geq rP+P,C)
\]
we obtain an improvement using the ABZ bound with choices $B=rP, Z=sQ$ (Theorem \ref{T:cbabz}).
As in the previous example we compare delta sets and find 
\begin{align*}
&I_P(0,9P+9Q) \backslash I_P(9Q,9P+9Q) = \{ k^- \in I_P(-9P,9Q) : k^- P \in \Gamma_P \} \\
= &\{ 141, 109, 77, (45), 137, 105, 73, 41, (9) \}. \\[1ex]
&I_P(9Q,9P+9Q) \backslash I_P(0,9P+9Q) = \{ k^+ \in I_P(9Q,9Q) : 9Q+k^+ P \not \in \Gamma_P \} \\
= &\{ 115, 147, 179, (211), 119, 151, 183, 215, (247) \}.
\end{align*}
The information shows that although the delta sets $\Delta_P(0,9P+9Q)$ and $\Delta_P(9Q,9P+9Q)$
have the same size, the first contains more divisors of small degree and the latter more divisors 
of high degree. For $Z=9Q$ and for $141 \leq r \leq 146$ (or $109 \leq r \leq 114$) we see that
\[
\# \Delta_P(\geq rP+P+sQ,C) - \# \Delta_P(\geq rP+P,C) = 5.
\]
The order bound gives minimum distance $d \geq 40$ for the AG code with $C=9P+9Q$ and
$d \geq 50$ for the AG code with $C=10P+9Q.$ Thus we improve the minimum distance for
$C=9P+9Q$ from $d \geq 40$ to $d \geq 45.$
\begin{align*}
\Delta_P(9P+9Q) \supseteq &\{ A_1 = 0, \ldots, A_{18} = 109P \} \\
 & \cup \{ A_{19} = 112P+9Q, \ldots, A_{45} = 256P+9Q \}
\end{align*}
\end{example}

\begin{example}
The ABZ bound, while more general than the order bound, is still only a special case
of the main theorem. The following choice of divisors in $\Delta_P(12P+12Q)$ gives
$\gamma_P(12P+12Q) \geq 56.$ This improves both the order bound and the ABZ bound
(for all possible choices of $A, B,$ and $Z$ as integer combinations of $P$ and $Q$).
\begin{align*}
\Delta_P(12P+12Q) \supseteq &\{ A_1 = 0, \ldots, A_{24} =116P \} \\
 & \cup \{ A_{25} = 118P+6Q, \ldots, A_{32} = 141P+6Q \} \\
 & \cup \{ A_{33} = 143P+12Q, \ldots, A_{56} = 259P+12Q \} 
\end{align*}
\end{example}
 
\appendix
 
\section{Coset decoding} \label{S:CosetDecoding}

For a given vector $y \in \ff^n$, and for an extension of linear codes $\Cc' \subset \Cc \subset \ff^n$, coset decoding determines
the cosets of $\Cc'$ in $\Cc$ that are nearest to the vector $y.$ If $y$ is at distance $d(y,\Cc) \leq t$ from $\Cc$ and
the minimum distance $d(\Cc/\Cc')$ between distinct cosets is at least $w > 2t$ then there exists a unique nearest coset 
$c + \Cc'$ with $d(y,c+\Cc') \leq t$. We describe a coset decoding procedure that returns the unique coset when the 
estimate $d(\Cc/\Cc') \geq w$ is obtained with Theorem \ref{T:cosetbound}. 
The procedure follows the majority coset decoding procedure in \cite{Duu93Thesis}, \cite{Duu93}. \\

Shift bound or Coset bound (Theorem \ref{T:cosetbound}):  
Let $\Cc/\Cc_1$ be an extension of $\ff$-linear codes
with corresponding extension of dual codes $\Cd_1/\Cd$ such that 
$\dim \Cc/\Cc_1 = \dim \Cd_1/\Cd = 1$.
If there exist vectors $a_1, \ldots, a_w$ and $b_1, \ldots, b_w$ such that 
\[
\begin{cases} 
a_i \ast b_j \in \Cd    &\text{for $i+j \leq w$}, \\
a_i \ast b_j \in \Cd_1 \backslash \Cd &\text{for $i+j = w+1$}, 
\end{cases}
\]
then $d(\Cc/\Cc_1) \geq w.$ \\

For a given $x \in \Cd_1 \backslash \Cd$, we may assume, after rescaling if necessary,
that $a_i \ast b_{w+1-i} \in x+\Cd$, for $i=1,\ldots,w.$ Define the following cosets of $a_i$ and $b_{w+1-i}$,
for $i=1,\ldots,w,$
\begin{align*} 
A_i &= a_i + \langle a_1, \ldots, a_{i-1} \rangle, \\
B_{w+1-i} &= b_{w+1-i} + \langle b_1, \ldots, b_{w-i} \rangle.
\end{align*}
For $c \in \Cc$, the coset $c+\Cc_1$ is uniquely determined by $x \cdot c$. 
For a given $y \in \ff^n$ such that $d(y,\Cc) \leq t$, the decoding procedure
will look for a pair $a' \in A_i, b' \in B_{w+1-i}$ such that, for
all $c \in \Cc$ with $d(y,c) \leq t$, $(a' \ast b') \cdot y = x \cdot c.$
The vector $a' \ast b'$ is defined as the Hadamard or coordinate-wise product
of the vectors $a'$ and $b'$. We use $(a' \ast b') \cdot y = (a' \ast y) \cdot b'$.

\begin{theorem} (Decoding up to half the coset bound) \label{T:decoding}
Let $2t < w \leq d(\Cc/\Cc_1)$, for $\Cc/\Cc_1$ and $w$ as in Theorem \ref{T:cosetbound}.
For $y \in \ff^n$ such that $d(y,\Cc) \leq t$, let
\begin{align*}
&I =  \{ 1 \leq i \leq w : \; (\exists a'_i \in A_i) \; (a'_i \ast b_j) \perp y, 1 \leq j \leq w-i \;\}, \\
&I^\ast =  \{ 1 \leq j \leq w : \; (\exists b'_j \in B_{w+1-j}) \; (a_i \ast b'_j) \perp y, 1 \leq i \leq j-1 \;\}. 
\end{align*}
For every $c \in \Cc$ with $d(y,c) \leq t$, $x \cdot c = (a'_i \ast b'_i) \cdot y$, for a majority of $i \in I \cap I^\ast.$
\end{theorem}
\begin{proof}
For $c \in \Cc$, let $a'_i \in A_i$ be such that $a'_i \ast y = a'_i \ast c.$ 
The vector $a'_i$, if it exists, satisfies $(a'_i \ast b_j) \cdot y = 0$, for $j = \{1,\ldots,w-i \}.$ Moreover, for
any $b' \in B_{w+1-i}$, $(a'_i \ast b') \cdot y =
(a'_i \ast b') \cdot c = (a_i \ast b_{w+1-i}) \cdot c = x \cdot c.$  Let
\[
\begin{array}{lcl}
\Gamma = \{ 1 \leq i \leq w : \; (\exists a'_i \in A_i) \;  a'_i \ast y = a'_j \ast c \; \}, &
&\Delta = \{ 1 \leq i \leq w \} \backslash \Gamma, \\
\Gamma^\ast = \{ 1 \leq j \leq w : \; (\exists b'_j \in B_{w+1-j}) \; b'_j \ast y = b'_j \ast c \}, &
&\Delta^\ast = \{ 1 \leq j \leq w \} \backslash \Gamma^\ast. 
\end{array}
\]
We know a priori only the sets $I$ and $I^\ast$. Clearly, $\Gamma \subset I$ and $\Gamma^\ast \subset I^\ast.$ 
Moreover, for $c \in \Cc$ with $d(y,c) \leq t$, $|\Delta|, |\Delta^\ast| \leq t.$ For $i \in I \cap I^\ast$, 
$(a'_i \ast b'_i) \cdot y = x \cdot c$ if either $i \in \Gamma$ or $i \in \Gamma^\ast$. Regardless of the actual sets 
$I$ and $I^\ast$, this is certainly the case if $i \in \Gamma \cap \Gamma^\ast$ and it fails only when 
$i \in \Delta \cap \Delta^\ast$. Now
\[
|\Gamma \cap \Gamma^\ast| - |\Delta \cap \Delta^\ast| = w - |\Gamma \cap \Delta^\ast| - |\Gamma^\ast \cap \Delta|
\geq w - 2t > 0.
\]
Thus, the majority of $i \in I \cap I^\ast$ will give a value $(a'_i \ast b'_i) \cdot y = x \cdot c.$ 
\end{proof}

If $\dim \Cc / \Cc' > 1$ then the procedure can be applied iteratively to a sequence of extensions
$\Cc' = \Cc_r \subset \Cc_{r-1} \subset \cdots \Cc_1 \subset \Cc_0 = \Cc$ such that $\dim \Cc_{i} / \Cc_{i-1} = 1$,
for $i=1,\ldots,r.$  For given $y_0 \in \ff^n$ with $d(y_0,\Cc_0) \leq t$, 
the procedure returns the unique coset $c_0+\Cc_1$ such that $d(y_0,c_0+\Cc_1) \leq t$. At the
next iteration, for $y_1 = y_0-c_0 \in \ff^n$ with $d(y_1,\Cc_1) \leq t$, the procedure returns the 
unique coset $c_1+\Cc_2$ such that $d(y_1,c_1+\Cc_2) \leq t$, and so on. \\

Let ${\cal A} = \{ A_1 \leq A_2 \leq \cdots \leq A_w \} \subset \Delta_P(C)$ be
a sequence of divisors with $A_{i+1} \geq A_i + P$, for $i=1,\ldots,w-1.$ 
Theorem \ref{T:cbdiv} (Main theorem) together with Lemma \ref{L:cosetgammas} shows that 
$d(C_\Omega(D,G-P) / C_\Omega(D,G)) \geq w$, for $G$ such that $C=G-K-P,$ and for
$D \cap (A_w-A_1) = \emptyset.$ We show how the coset decoding procedure applies to
the given extension. For a divisor $A_i \in \Delta_P(C)$, also $K+C+P-A_i=G-A_i \in \Delta_P(C).$ Thus, 
there exist functions $f_i \in L(A_i) \backslash L(A_i-P)$ and $g_i \in L(G-A_i) \backslash L(G-A_i-P).$
Let $(a_i \ast b_{w+1-j}) = ((f_i g_j)(P_n), \ldots, (f_i g_j) (P_n)),$ for $i \leq j.$ Then
\[
\begin{cases} 
a_i \ast b_j \in C_L(D,G) &\text{for $i+j \leq w$}, \\
a_i \ast b_j \in C_L(D,G) \backslash C_L(D,G-P) &\text{for $i+j = w+1$}, 
\end{cases}
\]
Moreover, we have the following interpretation for the sets $\Gamma, \Gamma^\ast, \Delta, \Delta^\ast.$ 
\[
\begin{array}{llcll}
i \in \Gamma &\Leftrightarrow~ A_i \in \Gamma_P(Q), &  &i \in \Delta &\Leftrightarrow~ A_i \in \Delta_P(Q), \\
i \in \Gamma^\ast &\Leftrightarrow~ A_i \in \Delta_P(C-Q), &  &i \in \Delta^\ast &\Leftrightarrow~ A_i \in \Gamma_P(C-Q).
\end{array}
\]
The order bound (Theorem \ref{T:order}) and the floor bound (Theorem \ref{T:floor}) as well as their generalizations
the ABZ bound for cosets (Theorem \ref{T:cbabz}) and the ABZ bound for codes (Theorem \ref{T:Floor2})
are all obtained in this paper as special cases of the main theorem. Thus, in each case coset decoding 
can be performed with Theorem \ref{T:decoding}. 

\nocite{Bra04} \nocite{BraSul07} \nocite{SKP07} \nocite{Sul01} \nocite{Pre98}
\nocite{HufPle03} \nocite{Lin99} 
\nocite{Pre98} \nocite{Ste99} 
\nocite{Sti93} 
\nocite{TsfVla07} 
\nocite{CraetSix05}
\nocite{HanSti90} 
\nocite{Mat04} 
\nocite{CarTor05} \nocite{MahMat06} \nocite{LunMcc06}
\nocite{Gar93} \nocite{KirPel95} 
\nocite{BeeNes06JPA} \nocite{Bee07FF} \nocite{HomKim06} \nocite{Kim94} \nocite{Mat01} 
\nocite{CheCra06} 
\nocite{Duu93}
\nocite{Duu08}

\newpage

\newcommand{\etalchar}[1]{$^{#1}$}
\def\lfhook#1{\setbox0=\hbox{#1}{\ooalign{\hidewidth
  \lower1.5ex\hbox{'}\hidewidth\crcr\unhbox0}}}


\begin{thebibliography}{CMdST07}

\bibitem[BA04]{Bra04}
Maria Bras-Amor{\'o}s.
\newblock Acute semigroups, the order bound on the minimum distance, and the
  {F}eng-{R}ao improvements.
\newblock {\em IEEE Trans. Inform. Theory}, 50(6):1282--1289, 2004.

\bibitem[BAO07]{BraSul07}
Maria Bras-Amor{\'o}s and Michael~E. O'Sullivan.
\newblock On semigroups generated by two consecutive integers and improved
  {H}ermitian codes.
\newblock {\em IEEE Trans. Inform. Theory}, 53(7):2560--2566, 2007.

\bibitem[Bee07]{Bee07FF}
Peter Beelen.
\newblock The order bound for general algebraic geometric codes.
\newblock {\em Finite Fields Appl.}, 13(3):665--680, 2007.

\bibitem[BT06]{BeeNes06JPA}
Peter Beelen and Nesrin Tuta{\c{s}}.
\newblock A generalization of the {W}eierstrass semigroup.
\newblock {\em J. Pure Appl. Algebra}, 207(2):243--260, 2006.

\bibitem[CC06]{CheCra06}
Hao Chen and Ronald Cramer.
\newblock Algebraic geometric secret sharing schemes and secure multi-party
  computations over small fields.
\newblock In {\em Advances in cryptology---{CRYPTO} 2006}, volume 4117 of {\em
  Lecture Notes in Comput. Sci.}, pages 521--536. Springer, Berlin, 2006.

\bibitem[CDG{\etalchar{+}}05]{CraetSix05}
Ronald Cramer, Vanesa Daza, Ignacio Gracia, Jorge~Jim{\'e}nez Urroz, Gregor
  Leander, Jaume Mart{\'{\i}}-Farr{\'e}, and Carles Padr{\'o}.
\newblock On codes, matroids and secure multi-party computation from linear
  secret sharing schemes.
\newblock In {\em Advances in cryptology---CRYPTO 2005}, volume 3621 of {\em
  Lecture Notes in Comput. Sci.}, pages 327--343. Springer, Berlin, 2005.

\bibitem[CDM00]{CraDamMau00}
Ronald Cramer, Ivan Damg{\aa}rd, and Ueli Maurer.
\newblock General secure multi-party computation from any linear secret-sharing
  scheme.
\newblock In {\em Advances in cryptology---{EUROCRYPT} 2000 ({B}ruges)}, volume
  1807 of {\em Lecture Notes in Comput. Sci.}, pages 316--334. Springer,
  Berlin, 2000.

\bibitem[CFM00]{CamFarMun00}
Antonio Campillo, Jos{\'e}~Ignacio Farr{\'a}n, and Carlos Munuera.
\newblock On the parameters of algebraic-geometry codes related to {A}rf
  semigroups.
\newblock {\em IEEE Trans. Inform. Theory}, 46(7):2634--2638, 2000.

\bibitem[CMdST07]{Caretal07}
C{\'{\i}}cero Carvalho, Carlos Munuera, Ercilio da~Silva, and Fernando Torres.
\newblock Near orders and codes.
\newblock {\em IEEE Trans. Inform. Theory}, 53(5):1919--1924, 2007.

\bibitem[CT05]{CarTor05}
C{\'{\i}}cero Carvalho and Fernando Torres.
\newblock On {G}oppa codes and {W}eierstrass gaps at several points.
\newblock {\em Des. Codes Cryptogr.}, 35(2):211--225, 2005.

\bibitem[Duu]{Duu93Thesis}
Iwan~M. Duursma.
\newblock {\em Decoding codes from curves and cyclic codes}.
\newblock Dissertation, Technische Universiteit Eindhoven, Eindhoven, 1993.

\bibitem[Duu93]{Duu93}
Iwan~M. Duursma.
\newblock Majority coset decoding.
\newblock {\em IEEE Trans. Inform. Theory}, 39(3):1067--1070, 1993.

\bibitem[Duuar]{Duu08}
Iwan~M. Duursma.
\newblock Algebraic geometry codes: general theory.
\newblock In C.~Munuera E.~Martinez-Moro and D.~Ruano, editors, {\em Advances
  in Algebraic Geometry Codes}, Series on Coding Theory and Cryptography. World
  Scientific, to appear.

\bibitem[FR93]{FenRao93}
Gui~Liang Feng and T.~R.~N. Rao.
\newblock Decoding algebraic-geometric codes up to the designed minimum
  distance.
\newblock {\em IEEE Trans. Inform. Theory}, 39(1):37--45, 1993.

\bibitem[GKL93]{Gar93}
Arnaldo Garc{\'{\i}}a, Seon~Jeong Kim, and Robert~F. Lax.
\newblock Consecutive {W}eierstrass gaps and minimum distance of {G}oppa codes.
\newblock {\em J. Pure Appl. Algebra}, 84(2):199--207, 1993.

\bibitem[HK06]{HomKim06}
Masaaki Homma and Seon~Jeong Kim.
\newblock The complete determination of the minimum distance of two-point codes
  on a {H}ermitian curve.
\newblock {\em Des. Codes Cryptogr.}, 40(1):5--24, 2006.

\bibitem[HP03]{HufPle03}
W.~Cary Huffman and Vera Pless.
\newblock {\em Fundamentals of error-correcting codes}.
\newblock Cambridge University Press, Cambridge, 2003.

\bibitem[HS90]{HanSti90}
Johan~P. Hansen and Henning Stichtenoth.
\newblock Group codes on certain algebraic curves with many rational points.
\newblock {\em Appl. Algebra Engrg. Comm. Comput.}, 1(1):67--77, 1990.

\bibitem[Kim94]{Kim94}
Seon~Jeong Kim.
\newblock On the index of the {W}eierstrass semigroup of a pair of points on a
  curve.
\newblock {\em Arch. Math. (Basel)}, 62(1):73--82, 1994.

\bibitem[KP95]{KirPel95}
Christoph Kirfel and Ruud Pellikaan.
\newblock The minimum distance of codes in an array coming from telescopic
  semigroups.
\newblock {\em IEEE Trans. Inform. Theory}, 41(6, part 1):1720--1732, 1995.
\newblock Special issue on algebraic geometry codes.

\bibitem[LM06]{LunMcc06}
Benjamin Lundell and Jason McCullough.
\newblock A generalized floor bound for the minimum distance of geometric
  {G}oppa codes.
\newblock {\em J. Pure Appl. Algebra}, 207(1):155--164, 2006.

\bibitem[Mat01]{Mat01}
Gretchen~L. Matthews.
\newblock Weierstrass pairs and minimum distance of {G}oppa codes.
\newblock {\em Des. Codes Cryptogr.}, 22(2):107--121, 2001.

\bibitem[Mat04]{Mat04}
Gretchen~L. Matthews.
\newblock Codes from the {S}uzuki function field.
\newblock {\em IEEE Trans. Inform. Theory}, 50(12):3298--3302, 2004.

\bibitem[MM06]{MahMat06}
Hiren Maharaj and Gretchen~L. Matthews.
\newblock On the floor and the ceiling of a divisor.
\newblock {\em Finite Fields Appl.}, 12(1):38--55, 2006.

\bibitem[O'S01]{Sul01}
Michael~E. O'Sullivan.
\newblock New codes for the {B}erlekamp-{M}assey-{S}akata algorithm.
\newblock {\em Finite Fields Appl.}, 7(2):293--317, 2001.

\bibitem[Par]{SKP07}
Seungkook Park.
\newblock {\em Applications of algebraic curves to cryptography}.
\newblock Dissertation, University of Illinois, Urbana, 2007.

\bibitem[Pre98]{Pre98}
Oliver Pretzel.
\newblock {\em Codes and algebraic curves}, volume~8 of {\em Oxford Lecture
  Series in Mathematics and its Applications}.
\newblock The Clarendon Press Oxford University Press, New York, 1998.

\bibitem[Ste99]{Ste99}
Serguei~A. Stepanov.
\newblock {\em Codes on algebraic curves}.
\newblock Kluwer Academic/Plenum Publishers, New York, 1999.

\bibitem[Sti93]{Sti93}
Henning Stichtenoth.
\newblock {\em Algebraic function fields and codes}.
\newblock Universitext. Springer-Verlag, Berlin, 1993.

\bibitem[TVN07]{TsfVla07}
Michael Tsfasman, Serge Vl{\u{a}}du{\c{t}}, and Dmitry Nogin.
\newblock {\em Algebraic geometric codes: basic notions}, volume 139 of {\em
  Mathematical Surveys and Monographs}.
\newblock American Mathematical Society, Providence, RI, 2007.

\bibitem[vL99]{Lin99}
J.~H. van Lint.
\newblock {\em Introduction to coding theory}, volume~86 of {\em Graduate Texts
  in Mathematics}.
\newblock Springer-Verlag, Berlin, third edition, 1999.

\bibitem[vLW86]{LinWil86}
Jacobus~H. van Lint and Richard~M. Wilson.
\newblock On the minimum distance of cyclic codes.
\newblock {\em IEEE Trans. Inform. Theory}, 32(1):23--40, 1986.

\end{thebibliography}

\end{document}